\documentclass[aos,preprint]{imsart}
\setattribute{journal}{name}{}

\RequirePackage[OT1]{fontenc}
\RequirePackage{amsthm,amsmath}
\RequirePackage[numbers]{natbib}
\usepackage{amsmath}
\usepackage{mathrsfs}
\usepackage{appendix}
\usepackage{float}
\usepackage{graphicx}
\usepackage{caption}
\usepackage{subcaption}
\usepackage{algorithm}
\usepackage{algpseudocode}
\usepackage{enumitem}
\usepackage{xcolor}
\usepackage{comment}

\startlocaldefs
\numberwithin{equation}{section}
\theoremstyle{plain}
\newtheorem{thm}{Theorem}
\newtheorem{thmdef}{Theorem}
\newtheorem{defn}[thmdef]{Definition}
\newtheorem{athmdef}{Theorem}
\numberwithin{athmdef}{section}
\newtheorem{adefn}[athmdef]{Definition}
\newtheorem{alemma}{Lemma}
\numberwithin{alemma}{section}

\newtheorem{thmlem}{Theorem}
\newtheorem{lemma}[thmlem]{Lemma}
\newtheorem{thmcor}{Theorem}
\newtheorem{corollary}[thmcor]{Corollary}
\endlocaldefs

\newcommand{\calO}{{\cal O}}
\newcommand{\calF}{{\cal F}}

\newcommand{\hSig}{\widehat \Sigma}
\newcommand{\tSig}{\widetilde \Sigma}

\newcommand{\tu}{\tilde u}
\newcommand{\trho}{\tilde\rho}

\newcommand{\tZ}{\widetilde Z}
\newcommand{\ty}{\widetilde y}
\newcommand{\tG}{\widetilde G}
\newcommand{\betah}{\hat \beta}
\newcommand{\lam}{\lambda}
\newcommand{\sigs}{\sigma^2}
\newcommand{\taus}{\tau^2}
\newcommand{\eps}{\epsilon}
\newcommand{\Siga}{\Sigma_A}
\newcommand{\linf}{\|_{\infty}}

\newcommand{\vhat}{\hat v}
\newcommand{\bI}{\mathrm I}
\newcommand{\dt}{\delta (\eps,\zeta)}

\begin{document}

\begin{frontmatter}
\title{CoCoLasso for High-dimensional Error-in-variables Regression}
\runtitle{Convex conditioned Lasso for Corrupted Data}
%\thankstext{T1}{Footnote to the title with the ``thankstext'' command.}

\begin{aug}
\author{\fnms{Abhirup} \snm{Datta}\ead[label=e1]{datta013@umn.edu}}
\and
\author{\fnms{Hui} \snm{Zou}
\ead[label=e2]{zouxx019@umn.edu}}

%\thankstext{t1}{Some comment}
%\thankstext{t2}{First supporter of the project}
%\thankstext{t3}{Second supporter of the project}
\runauthor{Datta and Zou}

\affiliation{University of Minnesota}%\thanksmark{m1} and Another University\thanksmark{m2}}

\address{Abhirup Datta\\
University of Minnesota\\
%Usually a few lines long\\
\printead{e1}}
%\phantom{E-mail:\ }\printead*{e2}}

\address{Hui Zou \\
University of Minnesota\\
%Usually a few lines long\\
%Usually a few lines long\\
\printead{e2}\\}
\end{aug}

\begin{abstract}
Much theoretical and applied work has been devoted to high-dimensional regression with clean data. However, we often face corrupted data in many applications where  missing data and measurement errors cannot be ignored.
Loh and Wainwright (2012) proposed a non-convex modification of the Lasso for doing high-dimensional regression with noisy and missing data. It is generally agreed that the virtues of convexity contribute fundamentally the success and popularity of the Lasso.
In light of this, we propose a new method named CoCoLasso that is convex and can handle a general class of corrupted datasets including the cases of additive measurement error and random missing data. 
We establish the estimation error bounds of CoCoLasso and its asymptotic sign-consistent selection property. We further elucidate how the standard cross validation techniques can be misleading in presence of measurement error and develop a novel corrected cross-validation technique by using the basic idea in CoCoLasso. The corrected cross-validation has its own importance. We demonstrate the superior performance of our method over the non-convex approach by simulation studies.
 %\vspace{-5mm}
\end{abstract}

\begin{keyword}[class=MSC]
\kwd[Primary ]{62J07}
\kwd[; secondary ]{62F12}
\end{keyword}

\begin{keyword}
\kwd{Convex optimization}
\kwd{Error in variables}
\kwd{High-dimensional regression}
\kwd{LASSO}
\kwd{Missing data}
\end{keyword}

\end{frontmatter}

\section{Introduction} High-dimensional regression has wide applications in various fields such as genomics, finance, medical imaging, climate science, sensor network, etc. The current 
inventory of high-dimensional regression methods includes Lasso \citep{lasso}, SCAD \citep{scad}, elastic net \citep{enet}, adaptive lasso \citep{adalasso} and Dantzig selector \citep{dantzig} among others. The articles \cite{fanli06} and \cite{fanlv10} provide an overview of these existing methods while the book by \cite{buhl} discusses their statistical properties in finer details. The canonical high-dimensional linear regression model assumes that the number of available predictors ($p$) is larger than the sample size ($n$), although the true number of relevant predictors ($s$) is much less than $n$. The model is expressed as $y=X\beta^* + w$ where $y=(y_1,\ldots,y_n)'$ is the vector of responses, $X=((x_{ij}))$ is the $n\times p$ matrix of covariates, $\beta^*$ is a $p\times1$ sparse coefficient vector with only $s$ non-zero entries and $w=(w_1,\ldots,w_n)'$ is the noise vector. 

Much of the existing theoretical and applied work on high-dimensional regression has focused on the clean data case. However, we often face corrupted data in many applications where the covariates  are observed inaccurately or have missing values. Common examples include sensor network data \citep{senserror}, high-throughput sequencing \citep{thru}, and gene expression data \citep{generror}. It is well known that misleading inference results will be obtained if the regression method for clean data is naively applied to the corrupted data. %How to handle measurement error is a fundamentally important topic in regression analysis, and there has been a vast amount of work on this topic in the classical low-dimensional regression context.  Interested readers are referred to \citep{Fuller87}, \citep{Carrollbook07} and references therein for the current literature on measurement errors.  
In order to facilitate further discussion, we assume that we observe a corrupted covariate matrix $Z=(z_{ij})_{1 \le i \le n, 1 \le j \le p}$ instead of the true covariate matrix $X$. Depending on the context, there can be various ways to model the measurement error. In the additive model setup, $z_{ij}=x_{ij}+a_{ij}$ where $A=(a_{ij})$ is the additive error matrix. In the multiplicative error setup, $z_{ij}=x_{ij}m_{ij}$ where $m_{ij}$s are the multiplicative errors. Missing predictors can be interpreted as a special case of multiplicative measurement errors with $m_{ij}=I(x_{ij} \mbox{ is not missing})$ where $I(\cdot)$ is the indicator function. 

Without loss of generality, we take the Lasso as an example to illustrate the impact of measurement errors. We apply the Lasso to the clean data by minimizing: 
\begin{equation}\label{eq:lasso_orig}
1/(2n)\| y - X\beta \| ^2_2 + \lambda \|\beta\| _1
\end{equation} with respect to $\beta$. Here $\lambda > 0$ is the regularization parameter and $\| \cdot \|_ p$ denotes the $\ell_p$ norm for vectors and matrices for $1\leq p \leq \infty$. If we ignore the measurement error issue, we would apply the Lasso to the corrupted data by minimizing:
\begin{equation}\label{eq:lasso_corrupt}
1/(2n) \| y -Z \beta \| ^2_2 + \lambda \|\beta\| _1.
\end{equation} 
However, as pointed out in \citep{rose10}, the resulting estimate of $\beta$ is often erroneous if the noise is large. We need to find a proper modification of (\ref{eq:lasso_corrupt}) such that its solution is comparable/close to the clean Lasso estimate  (\ref{eq:lasso_orig}).

Observe that the clean Lasso objective function can be equivalently formulated as
 \begin{equation}\label{eq:lasso00}
\frac 12 \beta'\Sigma\beta - \rho'\beta + \lam \|\beta \|_1 \mbox{ where } \Sigma = \frac 1n X'X \mbox{, } \rho = \frac 1n X'y.
\end{equation} 
In \citep{loh12} Loh and Wainwright use $Z$ and $y$ to construct unbiased surrogates $\hSig$ for  $\Sigma$ and $\trho$ for $\rho$. To elucidate, let us consider the classical additive measurement error case. 
Following \citep{loh12}, assume the additive errors $a_{ij}$ are independent with mean zero and variance $\tau^2$ where $\tau^2$ is a known constant, then 
\[
E[\frac 1n Z'Z]=\frac 1n X'X+\tau^2 \bI ,\quad E[\frac 1n Z'y-\frac 1n X'y]=0 .
\]
Thus Loh and Wainwright suggested using unbiased surrogates
\begin{equation}\label{eq:unbias}
\hSig=\frac 1n Z'Z-\tau^2\bI,  \quad \trho=\frac 1n Z'y
\end{equation} 
and then solve the following optimization problem to get an estimate of $\beta$:
 \begin{equation}\label{eq:lasso11}
\frac 12 \beta'\hSig\beta - \trho'\beta + \lam \|\beta \|_1 .
\end{equation}

Although the above solution is very natural, (\ref{eq:lasso11}) is fundamentally different from the clean Lasso. Notice that $\hSig$ may not be positive semi-definite. When  $\hSig$ does have a negative eigenvalue (which happens very often under high-dimensionality),  the objective function in (\ref{eq:lasso11}) is no longer  convex. Moreover, the objective function is unbounded from below when  $\hSig$ has a negative eigenvalue. To overcome these technical difficulties, Loh and Wainwright defined their estimator as
 \begin{equation}\label{eq:lasso22}
\hat \beta \in \underset{\|\beta \|_1 \le b_o \sqrt s }{\arg\min} \; \frac 12 \beta'\hSig\beta - \trho'\beta + \lam \|\beta \|_1 .
\end{equation} 
for some constant $b_0$. 
Note that ``$\in$" not ``$=$" is used in (\ref{eq:lasso22}) because the objective function may still have multiple local/global minimizers even within the region $\|\beta \|_1 \le b_o \sqrt s $. 
Through some careful analysis, Loh and  Wainwright showed that, if $b_0$ is properly chosen, a projected gradient descent algorithm will converge in polynomial time to a small neighborhood of the set of all global minimizers. 

In this article we propose the {\em Convex Conditioned Lasso (CoCoLasso)} --- a convex formulation of the Lasso that can handle a general class of corrupted datasets including the cases of additive or multiplicative measurement error and random missing data. 
CoCoLasso automatically enjoys the theoretical and computational benefits of convexity that contribute fundamentally to the success of the Lasso. Theoretically, we derive the statistical error bounds of CoCoLasso which are comparable to those given in \cite{loh12}.  Additionally, we establish the asymptotic sign-consistent selection property of CoCoLasso. Earlier \cite{oys13} derived asymptotic selection consistency properties for the estimator in (\ref{eq:lasso22}) only for the restrictive case of additive measurement error. However, our result does not require any specification of the type of measurement error. This is arguably the most general result for sign consistency in presence of measurement error. There is no sign-consistency result for the non-convex approach by Loh and  Wainwright.

Our method has another significant advantage over the non-convex approach by Loh and Wainwright. As mentioned earlier, choosing $b_0$ in (\ref{eq:lasso22}) is critically important to the estimator by Loh and Wainwright. Their theory requires  $b_0 \geq ||\beta^*\|_2$ in order to have desirable error bounds. Note that $\beta^*$ is unknown. %, this may not seem as restrictive as one can choose a very large $b_0$ to satisfy this constraint. 
On the other hand, $b_0$ cannot be too large due to the required lower-RE and upper-RE conditions. See Theorem 1 in (\ref{eq:lasso22}) for details. Therefore, in practice one has to carefully choose the $b_0$ value. Our method does not have this concern. From a pure practical viewpoint, our method uses one tuning parameter $\lambda$ while the non-convex approach needs two tuning parameters $b_0$ and $\lambda$. CoCoLasso can be readily solved by any efficient algorithm for solving the clean Lasso. For example, we can use the LARS algorithm \citep{lars} to efficiently compute the entire solution paths for CoCoLasso estimates as $\lam$ continuously varies. This is particularly useful for practitioners to understand the procedure. 

We notice that in the current literature little attention has been paid to the cross validation methods used for corrupted data. Simply replacing $Z$ by $X$ leads to biased version of the cross validation procedure (similar to (\ref{eq:lasso11}) being a biased version of (\ref{eq:lasso00})). We demonstrate how the ideas used to develop CoCoLasso can be seamlessly adapted to propose new corrected cross-validation technique tailored for data with measurement error. To our best knowledge, the existing work on high-dimensional regression with measurement error did not touch on this cross-validation issue. The new corrected cross-validation has its own independent importance.

It is worth pointing out that a Dantzig Selector type estimator named matrix uncertainty (MU)  estimator was proposed in \cite{rose10} for additive measurement error models.
An improved version of MU estimator was proposed in \cite{rose13}. \cite{linandconic} establishes near-optimal minimax properties of the estimator in \cite{rose13} and develops a conic-programming based estimator that achieves minimax bounds. Two more conic programming based estimators have been recently proposed in \cite{l1l2linf} for the same model setup. 
It has been empirically observed that solving the Lasso problem can be much faster than solving the Dantzig selector \cite{disc_dantzig_efron}.
Compared to Dantzig Selector type estimators and the conic programming based estimators, the direct Lasso-modification methods, such as CoCoLasso, would enjoy computational advantages, which is very important for high-dimensional data analysis.
 
The rest of the article is organized as follows. In Section \ref{sec:gen} we define the CoCoLasso estimator. In Section~\ref{sec:main} we discuss the main theoretical results. In Section \ref{sec:part} we discuss the consequences of the results in Section~\ref{sec:main} for additive and multiplicative measurement error setups. A new cross-validation technique for corrupted data is developed in Section~\ref{sec:cross}. In Section \ref{sec:sim} we present simulation results to demonstrate the empirical performance of CoCoLasso.

\section{CoCoLasso}\label{sec:gen}
%\subsection{Model and notations}\label{sec:model}
We first introduce some necessary notations and model setup. For any matrix $K=((k_{ij}))$, we write $K > 0$ ($\geq 0$) when it is positive (semi-)definite. Let $\| K \|_ \infty= \max_i \sum_j |k_{ij}|$ denote the matrix $\ell_ \infty$ norm whereas $\|K \|_{\max}= \max_{i,j} \; |k_{ij}|$ denote the elementwise maximum norm. Also let $\Lambda_{\min}(K)$ and $\Lambda_{\max}(K)$ denote the minimum and maximum eigen values of $K$ respectively. We assume that all variables are centered so that the intercept term is not included in the model and the covariance matrix $X$ has normalized columns i.e. $\frac 1n \sum_{i=1}^n x_{ij}^2 = 1$ for every $j=1,\ldots,p$.  Without loss of generality, assume that $S=\{1,2,\ldots,s\}$ is the true support set of the regression coefficient vector and write $\beta^*=(\beta^{*T}_S, 0')' $ and $X=\left( X_S,X_{S^c} \right)$. Hence the true model can be rewritten as $y=X_S\beta^*_S+w$ where the components of $\beta^*_S$ are non-zero. For any vector $v$, we can partition it as  $v=(v_S',v_{S^c}')'$. Also, we partition $\Sigma$ as \[ \Sigma = \left ( \begin{array}{cc} (1/n) X_S'X_S & (1/n)X_S'X_{S^c} \\(1/n)X_{S^c}'X_S & (1/n)X_{S^c}'X_{S^c} \end{array} \right) = \left ( \begin{array}{cc} \Sigma_{S,S}  & \Sigma_{S,S^c} \\ \Sigma_{S^c,S}  & \Sigma_{S^c,S^c} \end{array} \right)\]
In this work we consider the fixed design case because we want to avoid the identifiability issues between the true design matrix and the measurement error matrix. 
In the theoretical literature on the clean Lasso, it is often assumed that $w_i$'s are independent and identically distributed sub-Gaussian random variables with parameter $\sigma^2$. We use the same assumption here.

As mentioned earlier, in a clean setting where the predictor matrix $X$ is observed accurately, a Lasso estimate is obtained by minimizing (\ref{eq:lasso00}). When the dataset is corrupted by measurement errors, the observed matrix of predictors $Z$ is some function of the true covariance matrix $X$ and random errors. Based on $Z$ and $y$, estimates $\hSig$ and $\trho$ are constructed as surrogates to replace $\Sigma$ and $\rho$ respectively in (\ref{eq:lasso00}). Different pairs of unbiased estimates $(\hSig,\trho)$ are provided in \cite{loh12} for various types of measurement errors. We will present the actual form of $(\hSig,\trho)$ in section 4, but for now we only need to assume that $(\hSig,\trho)$  have been computed.

We now define a nearest positive semi-definite matrix projection operator as follows: for any square matrix $K$, 
\[
(K)_{+}=\arg\min_{K_1 \geq 0}  \|K-K_1 \|_{\max}.
\]
Then we denote $\tSig=(\hSig)_+$ and define our {\em Convex conditioned Lasso (CoCoLasso)} estimate as 
\begin{equation}\label{eq:canl}
\hat \beta = \underset{\beta}{\mbox{arg} \min} \; (1/2)\beta'\tSig\beta - \trho'\beta + \lam \|\beta\|_1
\end{equation}

We use an alternating direction method of multipliers (ADMM) \citep{boyd_admm} to obtain $\tSig$ from $\hSig$. The ADMM algorithm is very efficient and details of the algorithm are provided in Appendix \ref{app:admm}. By definition, $\tSig$ is always positive semi-definite. Note that $\Sigma$ is  positive semi-definite when $p>n$. Subsequently, we can reformulate our problem as:
\begin{equation}\label{eq:cocoaslasso}
\hat \beta = \underset{\beta}{\mbox{arg} \min} \; \frac 1n ||\ty-\tZ \beta||_2^2 + \lam \|\beta\|_1
\end{equation}
where $\tZ / \sqrt n $ is the Cholesky factor of $\tSig$ i.e. $\frac 1n \tZ'\tZ = \tSig$ and $\ty$ is such that $ \tZ'\ty = Z'y$.

Numerically, (\ref{eq:cocoaslasso}) is just like the clean Lasso. One can apply several very fast solvers to solve (\ref{eq:canl}), such as the coordinate descent algorithm \citep{glmnet} or the homotopy algorithm \citep{lars}. This is a great advantage for practitioners, as the Lasso solvers are widely used in practice and many know how to use them. We use the LARS-EN algorithm to obtain the solution as it simultaneously provides the entire solution path   for different values of $\lam$.

Theoretically, (\ref{eq:canl}) can be analyzed by the tools for analyzing the clean Lasso. The surrogate $\hSig$ chosen by \cite{loh12} is often an unbiased estimate of the true gram matrix $\Sigma$, achieving a desired rate of convergence under the max norm.
By definition, we have
\begin{equation}\label{Eq: trick}
\|\tSig-\Sigma \|_{\max} \leq \|\tSig-\hSig \|_{\max}+\|\hSig-\Sigma \|_{\max} \leq 2\|\hSig-\Sigma \|_{\max}
\end{equation}
Equation (\ref{Eq: trick}) ensures that $\tSig$ approximates $\Sigma$  as well as the initial surrogate $\hSig$. 

Compared with Loh and Wainwright's estimator in \cite{loh12},  CoCoLasso is guaranteed to be convex. This avoids the need of doing any non-convex analysis of the method. Furthermore, unlike \cite{loh12} our method does not require any knowledge of $\|\beta\|_1$ and thereby eliminates the need for an initial estimate to obtain a bound for $\|\beta\|_1$. In the next section, we show that CoCoLasso is sign consistent and has the $\ell_1,\ell_2$ error bounds comparable to that in \cite{loh12}.

\section{Theoretical Analysis}\label{sec:main} In this section we derive the $\ell_1$ and $\ell_2$ bounds for the statistical error of the CoCoLasso estimate as well as its support recovery probability bounds.

\subsection{$\ell_1$ and $\ell_2$ bounds for the statistical error}
We assume that $\hSig$ and $\trho$ are sufficiently `close' to $\Sigma$ and $\rho$ respectively in the following sense: 

\begin{defn} \label{def:close} Closeness condition: Let us assume that the distribution of  $\hSig$ and $\trho$ are identified by a set of parameters $\theta$. Then there exists universal constants $C$ and $c$, and positive functions $\zeta$ and $\eps_0$ depending on $\beta^*_S$, $\theta$ and $\sigs$ such that for every $\eps \leq \eps_0$, $\hSig$ and $\trho$ satisfy the following probability statements:
\begin{equation}\label{Eq: close}
\begin{array}{c}
Pr(|\hSig_{ij}-\Sigma_{ij} | \geq \eps) \leq C \exp \left(-cn\eps^2\zeta^{-1} \right) \; \forall \; i,j=1,\ldots,p \\
Pr(|\trho_j-\rho_j | \geq \eps) \leq C \exp \left(-cns^{-2}\eps^2\zeta^{-1} \right) \; \forall \; j=1,\ldots,p
\end{array}
\end{equation}
\end{defn}
The Closeness Condition requires that the surrogates $\hSig$ (and hence $\tSig$) and $\trho$ are close to $\Sigma$ and $\rho$ respectively in terms of the elementwise maximum norm. We show later in Section~\ref{sec:part} that this condition is satisfied by the surrogates defined in \cite{loh12} for commonly used additive or multiplicative measurement error models. %This implies that the modified objective function minimized to obtain $\betah$ is close to (\ref{eq:lasso_orig}) with high probability. 

We also assume the following compatibility or restricted eigenvalue condition:
\begin{equation}\label{eq:rec}
0 < \Omega = \quad \underset{x \neq 0, \, \|x_{S^c}\| _1 \leq 3 \|x_S\| _1}{min} \quad \frac {x'\Sigma x}{\|x\|_2^2}
\end{equation}
Restricted eigenvalue condition similar to this has been used in \cite{vdg09} to obtain bounds of statistical error of the clean Lasso estimate.

We now state the main result on the statistical error of the CoCoLasso estimate. All proofs are provided in Section \ref{sec:proofs}. Note that, for all the theoretical results, $C$ and $c$ denote generic positive constants. Their values vary from expression to expression but they remain universal constants.

\begin{thm}\label{th:error}
Under the assumptions stated in (\ref{Eq: close}) and (\ref{eq:rec}), for $\lam \leq min(\eps_0, 12\eps_0 \|\beta^*_S\linf )$ and $\eps \leq min(\eps_0, \Omega/64s)$ the following results holds true with probability at least $1-p^2C\exp\left(-cns^{-2}\lam^2\zeta^{-1} \right)-p^2C\exp\left(-cn\eps^2\zeta^{-1} \right)$:
\begin{align}\label{eq:errorbound}
\|\betah-\beta^*\|_2 \leq C \lam \sqrt s / \Omega \quad , \quad 
\|\betah-\beta^*\|_1 \leq C \lam s / \Omega
\end{align}
\end{thm}

Results similar to Theorem~\ref{th:error} were derived in Theorems 1 and 2 of \cite{loh12} for the estimates obtained by projected gradient descent algorithm for the non-convex objective function. Both the $\ell_1$ and $\ell_2$ bounds obtained in Theorem \ref{th:error}, are of the same order as the analogous bounds for statistical error of the traditional Lasso estimate. The tail probability depends on the presence of error in the variables through the component $\zeta$. Precise expression for $\zeta$ is derived for the case of additive  measurement error in Section \ref{sec:part}.

\subsection{Sign consistency}
In order to establish the sign consistency of CoCoLasso, in addition to the closeness conditions in (\ref{Eq: close}), we assume the irrepresentable and minimum eigenvalue conditions on $\Sigma$ which are sufficient and nearly necessary for sign consistency of the clean Lasso \citep{adalasso,zhaoyu,wainsgn}:
\begin{equation}\label{Eq: cond}
\| \Sigma_{S^c,S}\Sigma_{S,S}^{-1} \linf =1-\gamma <1 , \qquad  \Lambda_{\min}(\Sigma_{S,S}) = C_{\min} > 0 
\end{equation}
The main result on recovery of signed support is stated as follows:

\begin{thm}\label{Th: gen}
Under the assumptions given in Equations~(\ref{Eq: close}) and (\ref{Eq: cond}), for $\lam \leq min(\eps_0,4\eps_0/\gamma)$ and $\eps \leq \min(\eps_1,\lambda /(\lam \eps_2 +\eps_3))$ where $\eps_i$'s are bounded positive constants depending of $\Sigma_{S,S}$, $\beta^*_S$, $\theta$ and $\sigs$, the following occurs with probability at least $1-\delta_1$ where  $\delta_1 = p^2C\exp\left(-cns^{-2}\gamma^2\lam^2\zeta^{-1} \right)+p^2C\exp\left(-cns^{-2}\eps^2\zeta^{-1} \right)$

(a) There exists a unique solution $\betah$ minimizing (\ref{eq:canl}) whose support is a subset of the true support.

(b) $||\betah_S-\beta^*_S\linf \leq \kappa\lam $ where $\kappa =\left(4||\Sigma_{S,S}^{-1} \linf + C_{\min}^{-1/2} \right)$ 

(c) If $|\beta^*_{\min}| \geq  \kappa\lam $, then sign$(\betah_S)=$sign$(\beta^*_S)$
\end{thm}

If we assume for simplicity that $\kappa$ is $\calO(1)$ and the triplet $\{n,p,s\}$ and $\beta^*$ satisfy the scaling:
\begin{equation}\label{Eq: scale}
\begin{array}{c}
s^2\log p /n \rightarrow 0 \mbox{ as } n,p \rightarrow \infty \\
|\beta^*_{\min}| \gg s(\zeta \log p /n)^{1/2}
\end{array}
\end{equation}
then from the expression of $\delta_1$ in Theorem~\ref{Th: gen}  we can choose $\lambda$ so that $1- \delta_1$ goes to one, which implies the sign-consistency of the CoCoLasso estimate.

\begin{corollary}\label{Th: asym}
If $\Sigma$, $\tSig$ and $\trho$ satisfy the regularity conditions given in Theorem \ref{Th: gen}, then under the scaling in Equation (\ref{Eq: scale}), the CoCoLasso estimate $\betah$ defined in (\ref{eq:canl}) is sign-consistent if $|\beta^*_{\min}| \gg \lam \gg  s(\zeta \log p/n)^{1/2}$ and we also have the $\ell_{\infty}$ error bound $Pr(||\betah_S-\beta^*_S \linf \leq \kappa\lam) \rightarrow 1$.
\end{corollary}

So far in this section we have derived a general theory for the CoCoLasso where there is no assumption on the type of measurement error and the form of the estimates $\hSig$ and $\trho$. The only condition that requires a careful check is that the estimates $\hSig$ and $\trho$ are close enough to $\Sigma$ and $\rho$ respectively in the sense defined in (\ref{Eq: close}). In the next section, we consider two specific types of error-in-variables models and use the results of this section to derive the theoretical properties of CoCoLasso estimates for those models.

\section{CoCoLasso under Two Types of Measurement Errors}\label{sec:part}

\subsection{Additive error}\label{sec:add} We assume that the entries of the observed design matrix $Z$ is contaminated by additive measurement error i.e. $z_{ij}=x_{ij}+a_{ij}$ or in matrix notation, $Z=X+A$ where $A=((a_{ij}))$ is the matrix of measurement errors. We also assume that the rows of $A$ are independent and identically distributed with $0$ mean, finite covariance $\Siga$ and sub-Gaussian parameter $\taus$. Following \citep{loh12} we assume that $\Sigma_A$ is known. The unbiased estimates of $\Sigma$ and $\rho$ are given by $\hSig_{add}=\frac 1n Z'Z - \Sigma_A$ and $\trho_{add}=\frac 1n Z'y$, respectively. It is easy to observe that $\hSig_{add}$ can have negative eigenvalues precluding convex optimization. CoCoLasso estimates for this model will be based on the modified objective function 
$$\tilde f_{add}(\beta)=(1/2)\beta'\tSig_{add}\beta - \trho_{add}'\beta + \lam|\beta\|_1 \mbox{ where } \tSig_{add}=(\hSig_{add})_+.$$ 

The following results show that $\hSig_{add}$ and $\trho_{add}$ satisfy the conditions in Equation (\ref{Eq: close}). 
\begin{lemma}\label{Lem: addin}
$\hSig_{add}$ and $\trho_{add}$ satisfy the closeness conditions in (\ref{Eq: close}) with $\zeta=max(\taus,\taus\|\beta^*_S\|_ \infty ^2,\tau^4,\sigma^4)$ and $\eps_0=c\tau^2$.
\end{lemma}
%Pr( ||\tSig_{add}-\Sigma||_ {\max} \geq \eps) \leq  p^2C\exp\left(-c n\eps^2\zeta^{-1}\right) \label{Eq: bounddtilde} \\
So, even though $\hSig_{add}$ may not be positive definite, the surrogates $\hSig_{add}$ and $\trho$ satisfy (\ref{Eq: close}). The following result is an immediate consequence:

\begin{corollary}\label{cor:add} The results of Theorems \ref{th:error} and \ref{Th: gen} (and Corollary \ref{Th: asym}) hold for the CoCoLasso estimate for the additive error model under the assumptions (\ref{eq:rec}) and (\ref{Eq: cond}) (and (\ref{Eq: scale})) respectively.
\end{corollary}

As $\zeta$ increases with $\tau$ we see that the lower bound for $\lambda$ required in the Corollary increases as $\tau$ increases implying that more penalization is required in presence of larger measurement error to accurately recover the sparse support.

Note that the additive error covariance $\Sigma_A$ is assumed to be known in order to compute the CoCoLasso estimate. Similar assumption was used in \cite{loh12} and \cite{rose13} as it is unclear how to obtain a data-driven estimate of $\Sigma_A$ when only one dataset is available. If however, multiple replicates of the data are available,  following \cite{loh12}, one can obtain a data-driven estimate $\hSig_A$ of $\Sigma_A$ and define $\hSig_{add}=\frac 1n Z'Z - \hSig_A$. 

\subsection{Multiplictive error and missing data}\label{sec:mult}
If we assume that the errors are multiplicative, we observe $z_{ij}=x_{ij}m_{ij}$. In matrix notation, we have $Z=X\odot M$ where $M=((m_{ij}))$ and $\odot$ denotes the elementwise multiplication operator for vectors and matrices. We assume that the rows of $M$ are independent and identically distributed with mean $\mu_M$, covariance $\Sigma_M$ and sub-Gaussian parameter $\tau^2$. Under the assumption that the entries of $\mu_M$ and $\Sigma_M+\mu_M\mu_M'$ are strictly positive, \cite{loh12} suggests using the unbiased surrogates $\hSig_{mult}=(1/n) ZZ' \oslash (\Sigma_M+\mu_M\mu_M')$ and $\trho_{mult} = (1/n) Z'y \oslash \mu_M$ where $\oslash$ denotes the elementwise division operator for vectors and matrices. $\hSig_{mult}$ once again may not be positive semi-definite. The CoCoLasso estimate $\betah$ is obtained as
$$\min_{\beta} (1/2)\beta' (\tSig_{mult})_+\beta - \trho_{mult}'\beta + \lam|\beta\|_1 \mbox{ where } \tSig_{mult}=(\hSig_{mult})_+.$$

Randomly missing covariates can be formulated as a multiplicative error model. For example, a simple model assumes that $x_{ij}$'s are missing randomly with probability $r$ and their missing statuses are independent of one another. Then we can defining $z_{ij}=x_{ij}m_{ij}$ where $m_{ij}=I(x_{ij} \mbox{ is not missing }) \sim Bernoulli(1-r)$. Other missing data models with different choices of the missing probabilities (e.g. $m_{ij} \sim 
Bernoulli(1-r_j)$) will also fall under the same setup. We can obtain estimate of $r$ (or $r_j$) as the proportion of missing entries in the matrix (or in the $j^{th}$ column). For simplicity, we can assume $r$ is known and then $\Sigma_M$ and $\mu_M$ are known as well. 

We now establish analogous results for the CoCoLasso estimate in this multiplicative model setup. Note that as the errors are multiplicative, in order to have all the $z_{ij}$'s to be close to the respective $x_{ij}$'s, we need an upper bound for both $x_{ij}$ and $m_{ij}$. We also need a positive lower bound for the entries of $\mu_M$ and $\Sigma_M+\mu_M\mu_M'$ for the expressions of $\hSig_{mult}$ and $\trho_{mult}$ to be meaningful. To ensure these, we impose the following additional set of regularity conditions for the multiplicative setup:
\begin{equation}\label{Eq: multcond}
\begin{array}{cc}
\underset{i,j}{\max} \; |X_{ij}| = X_{\max} < \infty, \qquad & \underset{i,j}{\min} \;  E(m_1m_1') = M_{\min} > 0 \\
\min \; \mu_M=\mu_{\min} > 0, \qquad & \max\;  \mu_M = \mu_{\max} < \infty  \end{array}
\end{equation}

Under these regularity conditions the following lemma shows that $\tSig_{mult}$ and $\trho_{mult}$ satisfies the conditions in (\ref{Eq: close}):
\begin{lemma}\label{Lem: multin}
There exists positive functions $\eps_0$ and $\zeta$ depending on $\beta^*_S$, $\taus$, $\sigs$ and the constants in (\ref{Eq: multcond})  such that $\hSig_{mult}$ and $\trho_{mult}$ satisfy the closeness conditions in (\ref{Eq: close}).
\end{lemma}
Having proved Lemma \ref{Lem: multin}, once again we use Theorems \ref{th:error}, \ref{Th: gen} and Corollary \ref{Th: asym} to have the following results:
\begin{corollary}\label{cor:mult} The results of Theorems \ref{th:error} and \ref{Th: gen} (and Corollary \ref{Th: asym}) hold for the CoCoLasso estimate for the multiplicative error/missing data model under the assumptions (\ref{eq:rec}) and (\ref{Eq: cond}) (and (\ref{Eq: scale})) respectively.
\end{corollary}

%For i.i.d missing statuses, a very good estimator of $r$ is given by the proportion of observed entries in the design matrix $X$ and this estimate %$\hat r$ can be used to replace $r$ in the expressions of 
%$\Sigma_M,  \mu_M$  and these new surrogates will still satisfy the conditions in (\ref{Eq: close}). When $m_{ij} \sim Bernoulli(1-r_j)$, $\hat r_j$ %is given by the proportion of observed entries in the $j^{th}$ column of $X$.

\section{Corrected cross-validation}\label{sec:cross} 
In applications, cross-validation \citep{hastibs} is a widely used technique for choosing the tuning parameter in penalized methods. However, cross validation for data corrupted with measurement error has received very little attention. In the presence of noisy/corrupted data, naive application of cross-validation is biased and a novel correction is needed. To elucidate, consider the usual $K$-fold cross validation for selecting the tuning parameter in the clean Lasso. Let ($X_k$, $y_k$) denote the true design matrix and response vector for the $k^{th}$ fold of the data for $k=1,2,\ldots,K$. Likewise, let ($X_{-k}$, $y_{-k}$) denote the design matrix and response vector respectively after removing the $k^{th}$ fold. In absence of measurement error, the estimate for the prediction error for the $k^{th}$ fold is given by $err_k(\lambda)=\frac 1{n_k} \|y_k-X_k\betah_k(\lambda) \|_2^2$ where $n_k$ is the size of the $k^{th}$ fold and $\betah_k(\lambda)$ is the Lasso estimate based on $X_{-k}$, $y_{-k}$ with tuning parameter $\lambda$. The optimal $\lambda$ is obtained by minimizing the total cross-validation error, i.e., 
\begin{equation}\label{eq:cv}
\hat \lambda =  \underset{\lambda}{\arg \min}\; \frac 1K \sum_{k=1}^K \frac 1{n_k} \|y_k-X_k\betah_k(\lambda) \|_2^2.
\end{equation} 

However, when we face noisy/corrupted data, as $X$ is unknown or partially missing, (\ref{eq:cv}) is not directly available. If we naively use the observed data $(Z,y)$, then the cross-validated choice of $\lambda$ is defined by minimizing
\begin{equation}\label{eq:cvnaive}
 \frac 1K \sum_{k=1}^K \frac 1{n_k} \|y_k-Z_k\betah_k(\lambda) \|_2^2.
\end{equation}
Even when we use the CoCoLasso (or the estimator in \ref{eq:lasso22}) to compute $\betah_k(\lambda)$ based on $Z_{-k}$, $y_{-k}$, the above criterion is biased compared to (\ref{eq:cv}) in the same way the loss function in (\ref{eq:lasso11}) is a biased version of (\ref{eq:lasso00}).

Using simple algebra we observe that (\ref{eq:cv}) is equivalent to
\begin{equation}\label{eq:cv1}
\hat \lambda = \underset{\lambda}{\arg \min}\; \frac 1K \sum_{k=1}^K \betah_k(\lambda)'\Sigma_k\betah_k(\lambda) - 2 \rho_k'\betah_k(\lambda) ,
\end{equation} 
where $\Sigma_k=\frac 1{n_k} X_k' X_k $ and $\rho_k=\frac 1{n_k} X_k' y_k$.

It may seem that using the unbiased surrogates $\hSig_k$ and $\trho_k$ in (\ref{eq:cv1}) may overcome the bias issue. However, as $\hSig_k$ possibly has negative eigen values this will lead to a cross validation function unbounded from below.

In the light of the above discussion, we propose a new cross validation method for corrupted data that adapts the same central idea used to construct CoCoLasso i.e. we can use $(\hSig_k)_+$ and $\trho_k$ in (\ref{eq:cv1}). With this correction, the cross-validated $\lambda$ is defined as 
\begin{equation}\label{eq:newcv}
\tilde \lambda = \underset{\lambda}{\arg \min} \sum_{k=1}^K \betah_k(\lambda)'(\hSig_k)_+\betah_k(\lambda) - 2 \trho_k'\betah_k(\lambda) .
\end{equation} 
We call the above procedure the corrected cross-validation. 

\section{Numerical Studies}\label{sec:sim} 

We use simulated datasets to evaluate the performance of CoCoLasso. For comparison we also included the Loh and Wainwright's method described in \cite{loh12}. For convenience, we use NCL (Non-convex Lasso) to denote  Loh and Wainwright's method in this section.

\subsection{Simulation Models}\label{sec:adsim} 

We considered both additive measurement errors and multiplicative measurement errors in the simulation study.\\

{\bf Additive errors case}. We generate data from the model $y \sim N(X\beta^*,\sigma^2 I)$ where $$\beta^*=(3,1.5,0,0,2,0,\ldots,0)'.$$ The sample size $n$ is set to be $100$ and $p=250$. The rows of $X$ are independent and identically distributed normal random variables with mean zero and covariance matrix $\Sigma_X$. We consider two models for $\Sigma_X$ --- autoregressive ($\Sigma_{X,ij}=0.5^{|i-j|}$) and compound symmetry ($\Sigma_{X,ij}=0.5 + I(i=j)*0.5$). We set $\sigma=3$ giving a signal to noise ratio of $2.36$ for autoregressive (AR) and $3.20$ for compound symmetry (CS). We generate $Z=X+A$ where the rows of $A$ are independent and identically distributed $N(0,\tau^2 I)$ where $\tau=0.75$, $1$ and $1.25$. \\

{\bf Multiplicative Errors case}. We also evaluated the performance of CoCoLasso and NCL in a multiplicative errors setup. The true model is assumed to be same as in the additive error setup. We now generate $Z=X\odot M$ where we assume that the elements of $M=((m_{ij}))$ follow log-normal distribution i.e. $\log(m_{ij})$'s are independent and identically distributed $N(0,\tau^2)$ where $\tau=0.25$, $0.5$ and $0.75$. 

\subsection{Simulation results and conclusions}
We used 5-fold corrected cross-validation for the CoCoLasso in our numerical examples. The code for NCL was provided by Dr. Po-Ling Loh.  NCL requires an initial estimator. Following \cite{oys13}, the initial estimate is a naive Lasso estimate based on $y$ and $Z$ which is tuned by 5-fold cross validation. 
NCL also requires knowledge of $\|   \beta_ S ^* \| _1$ for choosing the constraint parameter. Since, this is impossible to know beforehand, a naive 5-fold cross validation was used to select the optimal $R$ from 100 equally spaced values in $[R_{max}/500, 2*R_{max}] $ where $R_{max}$ is the $\ell_1$ norm of the initial estimate.

The accuracy of estimators is gauged by the Prediction Error (PE) and the Mean Squared Error (MSE) where
$$PE(\betah)= (\beta^*-\betah)'\Sigma_X (\beta^*-\betah)$$
and 
$$MSE(\betah)=\|\beta^*-\betah\|_2^2.$$ 
To evaluate variable selection, we record $C$ and $IC$ that denote the number of correct and incorrect predictors identified,  respectively. %The model comparison metrics for the two methods are presented in Table~\ref{tab:add}. 

Table~\ref{tab:add} and Table~\ref{tab:mult} summarize the simulation results for the additive error case and the multiplicative error case, respectively.
\begin{table}[h]
\centering
\caption{Summary statistics for the additive error simulation study based on 100 replications. Reported numbers are the medians and standard errors ($se$) are computed by bootstrap. ``CoCo" stands for CoCoLasso. ``NCL" is the method in Loh and Wainwright (2012). AR denotes Autoregressive covariance for the predictors whereas CS denotes compound symmetry covariance. %.``Clean Lasso"  is the lasso estimate by using $y$ and $X$ (the clean data).
 }\label{tab:add}
\renewcommand{\arraystretch}{1}
\begin{tabular}{c|c|cc|cc|cc} \hline 
&  & \multicolumn{2}{c|}{$\tau=0.75$}  & \multicolumn{2}{c|}{$\tau=1.0$} & \multicolumn{2}{c}{$\tau=1.25$}\\ \hline
&  & CoCo & NCL & CoCo & NCL & CoCo & NCL\\ \hline
 & C & 3 & 3 & 3 & 2 & 3 & 2 \\
 & IC & 11 & 3 & 11 & 1 & 10 & 0 \\
AR & PE & 3.66 & 4.13 & 5.8 & 6.91 & 8.49 & 10.92 \\
 & se(PE) & 0.19 & 0.26 & 0.26 & 0.34 & 0.5 & 0.46 \\
 & MSE & 3.81 & 3.76 & 5.57 & 6.07 & 7.94 & 8.36 \\
 & se(MSE) & 0.19 & 0.18 & 0.2 & 0.27 & 0.24 & 0.3 \\
&&&&&&& \\
 & C & 2 & 2 & 2 & 1.5 & 2 & 1\\
 & IC & 14 & 11.5 & 18 & 7 & 21 & 5\\
CS & PE & 4.49 & 4.57 & 6.03 & 6.91 & 6.99 & 10.47\\
 & se(PE) & 0.22 & 0.31 & 0.22 & 0.34 & 0.25 & 0.58\\
 & MSE & 8.05 & 8.03 & 11.01 & 10.31 & 12.97 & 15.06\\
 & se(MSE) & 0.33 & 0.48 & 0.37 & 0.99 & 0.34 & 1.06 \\ \hline
\end{tabular}
\end{table}
We observe that  CoCoLasso is more accurate than NCL as measured by PE and MSE, and the gap between the two methods widens as the perturbation level increases (measured by $\tau$).
NCL tends to select a sparser model than CoCoLasso, it tends to miss importance variables as the noise level is high. 

\begin{table}[!t]
\centering
\caption{Summary statistics for the multiplicative error simulation study based on 100 replications. Reported numbers are the medians and standard errors ($se$) are computed by bootstrap. ``CoCo" stands for CoCoLasso. ``NCL" is the method in Loh and Wainwright (2012). AR denotes Autoregressive covariance for the predictors whereas CS denotes compound symmetry covariance. %``Clean Lasso"  is the lasso estimate by using $y$ and $X$ (the clean data).
 }\label{tab:mult}
\renewcommand{\arraystretch}{1}
\begin{tabular}{c|c|cc|cc|cc} \hline 
&  & \multicolumn{2}{c|}{$\tau=0.25$}  & \multicolumn{2}{c|}{$\tau=0.5$} & \multicolumn{2}{c}{$\tau=0.75$}\\ \hline
&  & CoCo & NCL & CoCo & NCL & CoCo & NCL\\ \hline
 & C & 3 & 3 & 3 & 3 & 3 & 2\\
 & IC & 14 & 12 & 12 & 6 & 10 & 1\\
AR & PE & 2.02 & 2.47 & 3.25 & 3.58 & 7.32 & 8.32\\
 & se(PE) & 0.15 & 0.18 & 0.14 & 0.25 & 0.2 & 0.29\\
 & MSE & 1.95 & 2.26 & 2.93 & 3.09 & 6.19 & 6.58\\
 & se(MSE) & 0.09 & 0.14 & 0.14 & 0.18 & 0.2 & 0.26\\
&&&&&&& \\

 & C & 3 & 3 & 3 & 3 & 2 & 1\\
 & IC & 15 & 18 & 13 & 11 & 16 & 4\\
CS & PE & 2.23 & 2.37 & 3.66 & 3.82 & 7.93 & 9.31\\
 & se(PE) & 0.16 & 0.1 & 0.15 & 0.19 & 0.3 & 0.41\\
 & MSE & 4.21 & 4.32 & 6.11 & 5.75 & 10.43 & 9.34\\
 & se(MSE) & 0.27 & 0.21 & 0.27 & 0.26 & 0.25 & 0.61\\ \hline
\end{tabular}
\end{table}

\section{Summary}
In this paper we have proposed a novel convex approach to modify the classical Lasso with the clean data to handle the noisy data case. Our approach, named CoCoLasso, is easy to understand, easy to use and has solid theoretical foundations. We also have devised a novel cross validation methods for corrupted data. We have demonstrated the superior performance of our method over the non-convex approach in Loh and Wainwright (2012) by simulation studies. 

Finally, we would like to comment on the generality of the CoCoLasso approach.  Although we use the Lasso to illustrate the idea of CoCoLasso, the basic approach of CoCoLasso can be directly used in conjunction with other popular convex penalized methods. For example,  the fused Lasso \cite{fusedlasso} is a popular technique for ordered variable selection. Following the development of CoCoLasso, we can readily develop  CoCo-FusedLasso. We opt not to discuss these variants in the present paper.

\section{Proofs}\label{sec:proofs}
In this section we present the proofs of Theorems \ref{th:error} and \ref{Th: gen} as well as Lemmas \ref{Lem: addin} and \ref{Lem: multin}. A few useful properties and technical results about sub-Gaussian random variables required in the proofs are provided in Appendix~\ref{app:subg}. Throughout this section we denote $C$ and $c$ to be universal constants whose values may vary across different expressions. We also introduce a few additional notations used subsequently in the proofs.\begin{equation}\label{Eq: extracond}
\begin{array}{cccc}
D = \tSig - \Sigma,  \quad & G = \Sigma_{S^c,S}\Sigma_{S,S}^{-1}, \quad & \tG = \tSig_{S^c,S}\tSig_{S,S}^{-1}, \quad & H=\tG-G\\
F=\tSig_{S,S}^{-1}-\Sigma_{S,S}^{-1}, \quad & \phi = ||\Sigma_{S,S}^{-1} ||_ \infty,  \quad & \psi=||\Sigma_{S,S} ||_ \infty, \quad & B=||\beta^*_S \linf 
\end{array}
\end{equation}

\subsection{Proof of Theorem \ref{th:error}}\label{sec:therror} We first state and prove a simple result which will be later used in the proof:

\begin{lemma}\label{lem:trick} For any $\epsilon > 0$ we have,
\begin{equation}\label{eq:prick}
Pr(\| \tSig-\Sigma \|_ {\max} \geq \eps) \leq p^2 \max_{i,j} Pr(|\hSig_{ij}-\Sigma_{ij} | \geq \eps/2)
\end{equation}
\end{lemma}

\begin{proof}
From Equation (\ref{Eq: trick}) we have \[
Pr(\| \tSig-\Sigma \|_ {\max} \geq \eps) \leq Pr(\| \hSig-\Sigma \|_ {\max} \geq \eps/2) \]
The proof then follows using union bounds over $Pr(|\hSig_{ij}-\Sigma_{ij} | \geq \eps/2)$.
\end{proof}

\begin{proof} [Proof of Theorem \ref{th:error}]
The general idea of the proof closely resembles the proofs of \cite[Lemma 6.3 and Theorem 6.1]{buhl} for obtaining the error bounds of the traditional Lasso estimate. From the definition of $\betah$ in (\ref{eq:canl}), we have
\begin{align*}
\frac 12 \betah'\tSig\betah - \trho'\betah + \lam \|\betah \|_1 \leq \frac 12 \beta^{*T}\tSig\beta^* - \trho'\beta^* + \lam \|\beta^* \|_1
\end{align*}
Expanding $\betah$ as $\vhat+\beta^*$ where $\vhat=\betah-\beta^*$, this simplifies to 
\begin{equation}\label{eq:orineq1} \frac 12 \vhat'\tSig\vhat + \lam \|\betah\|_1  \leq \vhat'(\trho-\tSig\beta^*)+\lam\|\beta^*\|_1 
 \leq ||\vhat\|_1 ||\trho-\tSig\beta^* \|_ \infty + \lam\|\beta^*\|_1 
\end{equation}
In order to obtain an upper bound for the left hand side we first bound the quantity $\|\trho-\tSig\beta^*\|_\infty$. Using triangular inequality we have
\begin{align*}
\|\trho - \tSig\beta^* \|_\infty \leq \|\trho-\rho\linf+ \|\rho-\Sigma\beta^*\linf+\|D\beta^*\linf
\end{align*}
Using union bounds on the second equation of (\ref{Eq: close}) we see that for $\lam \leq 6\eps_0$, we have $P(\|\trho-\rho\linf > \lam/6) \leq pC\exp\left(-ncs^{-2}\lam^2\zeta^{-1}\right)$. As $\|D\beta^*\linf \leq sB\|D\|_{\max}$, Lemma \ref{lem:trick} alongwith the first equation of (\ref{Eq: close}) implies that for $\lam \leq 12B\eps_0$, $P(sB\|D\|_{\max} > \lam/6) \leq p^2C\exp\left(-ncs^{-2}\lam^2\zeta^{-1}B^{-2}\right)$. The third component $\rho-\Sigma\beta^* = \frac 1n X'w$ is a linear combination of independent sub-Gaussian errors $w$. As the columns of $X$ are normalized, invoking property \ref{eq:hoeff}, we have $P(\|\rho-\Sigma\beta^*\linf > \lambda/6) \leq pC\exp\left(-nc\lam^2\sigma^{-2}\right)$. Redefining $\zeta=max(\zeta, B^2\zeta, \sigma^2)$ we have 
\begin{align*}
\|\trho-\tSig\beta^*\linf < \lambda/2 \mbox{ on } {\calF} \mbox{ where } P({\calF}) \geq 1-p^2C\exp\left(-ncs^{-2}\lam^2\zeta^{-1}\right)
\end{align*}
For the remainder of the proof we restrict ourselves to $\calF$ adjusting for the probability of $\calF^c$. Returning to Equation (\ref{eq:orineq1}), we now have on $\calF$,
\begin{align*}
\frac 12 \vhat'\tSig\vhat + \lam \|\betah\|_1  \leq \frac \lam 2 ||\vhat\|_1 + \lam\|\beta^*\|_1 
\end{align*}
Since $\beta^*_{S^c}=0$, we know that $\vhat_{S^c}=\betah_{S^c}$, $\|\beta^*\|_1=\|\beta^*_S\|_1$. Also for any vector $x$, we can write $\|x\|_1=\|x_S\|_1+\|x_{S^c}\|_1$. Combining these, we have: 
\begin{align*}
\frac 12 \vhat'\tSig\vhat + \lam \|\betah_S\|_1 + \lam \|\vhat_{S^c}\|_1 \leq \frac \lam 2 ||\vhat_S\|_1 + \frac \lam 2 ||\vhat_{S^c}\|_1 + \lam\|\beta^*_S\|_1 
\end{align*}
Using the fact that $\|\betah_S\|_1 \geq \| \beta^*_S\|_1 - \| \vhat_S \|_1 $, we now have 
\begin{align}\label{eq:orineq2}
\vhat'\tSig\vhat + \lam \|\vhat_{S^c}\|_1 \leq 3\lam ||\vhat_S\|_1 
\end{align}
As $\vhat'\tSig\vhat \geq 0$, we have that on $\calF$, $\|\vhat_{S^c}\|_1 \leq 3 ||\vhat_S\|_1$. The Restricted Eigenvalue Condition (\ref{eq:rec}) immediately implies that on $\calF$, $\vhat'\Sigma\vhat \geq \Omega \|\vhat \|_2^2 $. Now 
\begin{align*} 
\vhat' \Sigma\vhat + \lam \|\vhat\|_1 & = \vhat'\tSig\vhat + \lam \|\vhat_{S}\|_1 + \lam \|\vhat_{S^c}\|_1 + \vhat' D \vhat\\ 
& \leq 4\lam \|\vhat_{S}\|_1 + \vhat' D \vhat \mbox{ using Eqn. (\ref{eq:orineq2}) } \\ 
& \leq 4\lam \sqrt s \|\vhat_{S}\|_2 + \vhat' D \vhat \leq 4\lam \sqrt s \|\vhat \|_2 + \vhat' D \vhat \\
& \leq 4\lam \sqrt {s} \sqrt {\frac {\vhat'\Sigma\vhat}{\Omega}} + \vhat' D \vhat \mbox{ using condition (\ref{eq:rec}) }  \\
& \leq \frac {\vhat'\Sigma\vhat}{4} +  \frac {16\lam^2 s}\Omega + |\vhat' D \vhat| \mbox{ using } 4ab \leq a^2/4 + 16b^2
\end{align*}
The last term on the right hand side is bounded as follows, 
\begin{align*}
|\vhat' D \vhat| & \leq \|D\|_{\max} \|\vhat\|_1^2 = \|D\|_{\max} (\|\vhat_S\|_1+\|\vhat_{S^c}\|_1)^2  \leq 16 ||D\|_{\max} \|\vhat_S\|_1^2 \mbox{ on } \calF \\
& \leq 16s||D\|_{\max} \|\vhat_S \|_2^2 \leq 16s||D\|_{\max} \|\vhat \|_2^2
\end{align*}
Using Lemma \ref{lem:trick} and the closeness condition (\ref{Eq: close}), for $\eps \leq min(\eps_0, \Omega/64s)$, 
\begin{align*}
P(16s\|D\|_{\max} > \Omega/4) = P(\|D\|_{\max} > \Omega/64s) \leq p^2C\exp\left(-nc\eps^2\zeta^{-1} \right)
\end{align*}
With probability at least $1-p^2C\exp\left(-nc\eps^2\zeta^{-1} \right)-p^2C\exp\left(-ncs^{-2}\lam^2\zeta^{-1} \right)$ we now have
\begin{align*}
\vhat' \Sigma\vhat + \lam \|\vhat\|_1 \leq \frac {\vhat'\Sigma\vhat}{4} +  \frac {16\lam^2 s}{\Omega} + \frac \Omega 4 \|\vhat \|_2^2
\end{align*}
One more application of the restricted eigenvalue condition (\ref{eq:rec}) now yields
\begin{align*}
\frac \Omega 2 \|\vhat\|_2^2 + \lam \|\vhat\|_1 \leq \frac {16\lam^2 s}{\Omega}
\end{align*}
which proves the bounds for both the $\ell_1$ and $\ell_2$ errors in Theorem \ref{th:error}
\end{proof}

\subsection{Proof of Theorem \ref{Th: gen}}\label{th:sign} The proof for the sign consistency result of the CoCoLasso is involved. We first present a series of results required to prove Theorem \ref{Th: gen}.

\begin{lemma} \label{Lem: duality} Let $\partial||x||_1$ denotes the sub-gradient of $||x||_1$ for any vector $x$. Then we have the following results:
(a) $\betah$ is the optimal solution to $\tilde f (\beta) = (1/2)\beta'\tSig\beta - \trho'\beta + \lam|\beta\|_1$ iff there exists a vector $ \tu $ in $\partial \|\betah \|_ 1$ such that
\begin{equation}\label{Eq: gradient}
\tSig \betah - \trho +\lam \tu = 0
\end{equation}
(b) If $|\tu_j| < 1 \; \forall j \in S^c$, then any other optimal solution $\tilde \beta$ will have support $S(\tilde \beta) \subseteq S$
(c) If we assume that $\tSig_ {S(\hat\beta),S(\hat\beta)}$ is invertible then under the conditions of part (b), $\tilde f(\beta)$ has unique minima
\end{lemma}

\begin{proof}
This lemma is a modified version of \cite[Lemma 1]{wainsgn}. We omit the proof as it is exactly analogous to that in the paper. 
\end{proof}

Note that the invertibility assumption of part (c) of Lemma~\ref{Lem: duality} needs to hold to establish the uniqueness of the Lasso solution. We now show that this occurs with probability tending to 1. For notational convenience, we define:
\begin{equation}
 \dt=p^2 C \exp\left(-cns^{-2}\eps^2\zeta^{-1}\right)
\end{equation}

\begin{lemma}\label{Lem: inv}
$Pr(\tSig_{S,S} > 0) \geq 1-\dt$ for all $\eps \leq min(\eps_0,C_{\min}/2)$
\end{lemma}

\begin{proof}
From Equation (\ref{Eq: extracond}), we have
\begin{align*}
 \Lambda_{\min}(\tSig_{S,S}) \geq & \Lambda_{\min}(\Sigma_{S,S})- |\Lambda_{\max}(-D_{S,S})|  \geq C_{\min}-||D_{S,S}||_2 \\
 \geq& C_{\min} - s ||D_{S,S} ||_ {\max} \geq C_{\min} - s ||D ||_ {\max} \geq C_{\min}/2
\end{align*}
where the last inequality occurs with probability at least $1-\dt$ for $\eps \leq min(\eps_0,C_{\min}/2)$
\end{proof}

% We show in the following theorem, that for the linear model setup described in Section \ref{Sec: method}, under certain conditions on $\tSig$ and $\trho$ these conditions are satisfied with probability tending to $1$ as $n$ goes to infinity. %We have relegated all the proofs and accessory results to Appendix \ref{A1}.
%To prove Theorem \ref{Th: gen}, we now state a series of results in the following lemmas assuming the conditions on $\tSig$ and $\trho$ given in Equation \ref{Eq: close}. 
\begin{lemma}\label{Lem: extrabounds}
If $\hSig$ and $\trho$ satisfy (\ref{Eq: close}), then there exists positive constants $C$, $c$ such that for every $\eps \leq min(\eps_0,1/\phi)$, 
\begin{equation}\label{Eq: extrabounds}
\begin{array}{c}
Pr(|| F ||_ \infty \geq \eps\phi^2(1-\phi\eps)^{-1}) \leq \dt  \\
Pr(|| H ||_ \infty \geq \eps\phi(2-\gamma)(1-\phi\eps)^{-1}) \leq \dt
\end{array}
\end{equation}
\end{lemma}
\begin{proof} Let $\eta_1 = || D_{S,S} \linf$ and $\eta_2 = || D_{S^c,S} \linf $. Now, $\sum_{j=1}^s |D_{ij}| \leq s||D||_{\max} $ for $(i=1,\ldots,s)$. Consequently, if $||D||_{\max} \leq \eps/s$ then both $\eta_1$ and $\eta_2$ are less than $\eps$. From (\ref{Eq: close}) and (\ref{eq:prick}), $Pr(\eta_1 \leq \eps, \eta_2 \leq \eps) \geq 1-\dt$ for $\eps \leq \eps_0$. The remainder of the proof follows from \cite[Lemma A2]{huisda}.
\end{proof}

\begin{proof}[Proof of Theorem \ref{Th: gen} Part (a)]
We use a Primal Dual Witness construction technique similar to \cite{wainsgn} to prove Theorem \ref{Th: gen}. Let $\betah_S$ be the solution to the restricted modified Lasso program i.e. 
\begin{equation}\label{Eq: rest betahat}
\betah_S = \underset {\beta_S}{arg\; min} \; \tilde  f_S(\beta_S) \mbox{ where } \tilde  f_S(\beta_S)= \frac 12 \beta_S'\tSig_{S,S}\beta_S - \trho_S'\beta_S + \lam \|\beta_S \|_1
\end{equation}
Let $\betah = (\betah_S',0_{(p-s)\times 1}')'$ and $\tu=(\tu_S',\tu_{S^c}')'$ where $\tu_S \in \partial(||\betah_S||_1)$ and $\tu_{S^c}$ is some unspecified $(p-s)\times 1$ vector. From part (a) of Lemma \ref{Lem: duality}, we observe that $\betah$ is an optimal solution to (\ref{eq:canl}) iff $\{\betah,\tu\}$ satisfies:
\begin{equation}\label{Eq: tb}
\begin{array}{c}
\tSig_{S,S} \betah_S - \trho_S +\lam \tu_S = 0 \\
\tSig_{S^c,S}\betah_S - \trho_{S^c} +\lam \tu_{S^c} = 0
\end{array}
\end{equation}
Solving for $\betah_S$ and $\tu_{S^c}$ from Equation~(\ref{Eq: tb}) we have:
\begin{equation}\label{Eq: dualfeas vector}
\begin{array}{c}
\betah_S=\tSig_{S,S}^{-1}(\trho_S-\lam \tu_{S} ), \qquad 
\tu_{S^c} = \tG\tu_S + \frac 1{\lam}\left (\trho_{S^c} - \tG\trho_S\right)
\end{array}
\end{equation}

From parts (b) and (c) of Lemma~\ref{Lem: duality}, we see that $\betah$ will be the unique solution to (\ref{eq:canl}) if $\tSig_{S,S} $ is non-singular and all the entries of $\tu_{S^c}$ have absolute values less than 1. Lemma~\ref{Lem: inv} provides lower bounds for $Pr(\tSig_{S,S} > 0)$. We now derive the bounds for $Pr(||\tu_{S^c}||_\infty < 1)$. We expand $\tu_{S^c}$ as :
\begin{align*} \tu _{S^c} = & \; G\tu _S + H\tu _S + \frac 1{\lam}\left((\trho_{S^c}-\rho_{S^c})+(\rho_{S^c}-G\rho_S) +  G (\rho_{S}-\trho_S)-H\trho_S \right) \\
=& \; G\tu _S + H\left( \tu_s + \frac 1\lambda  (\rho_S-\trho_S) - \frac 1\lambda \rho_S \right) \\
&+ \frac 1{\lam}\left((\trho_{S^c}-\rho_{S^c})+(\rho_{S^c}-G\rho_S) +  G (\rho_{S}-\trho_S)\right)
\end{align*}
Taking the absolute values and using triangular inequalities, we have:
\begin{align*}
||\tu _{S^c}||_ \infty & \leq ||G\tu_S||_\infty + ||H||_ \infty  \left(1 + \frac 1{\lam} ||\trho_S-\rho_S||_ \infty + \frac 1{\lam} ||\rho_S||_ \infty \right) \\ 
&  \frac 1{\lam} ||\rho_{S^c}-G\rho_S||_ \infty + \left( \frac 1{\lam}||\trho_{S^c}-\rho_{S^c}||_ \infty  + \frac 1{\lam}||G(\tilde \rho_{S}-\rho_{S})||_\infty \right)
\end{align*}

We bound each of the four terms on the right hand side separately. The irrepresentable condition (\ref{Eq: cond}) implies that $||G\tu_S||_\infty < (1-\gamma)$. It also implies that for $\lam \leq 4\eps_0/\gamma$ we have:
\begin{align*}
 Pr(\frac 1{\lam}&||\trho_{S^c}-\rho_{S^c}||_ \infty  + \frac 1{\lam}||G(\tilde \rho_{S}-\rho_{S})||_\infty < \gamma/2 ) \\
 &\geq Pr\left(\frac 1\lambda ||\trho-\rho\linf < \gamma/4\right) \geq 1-\delta(\lambda \gamma, \zeta)
 \end{align*}
where the last inequality follows from taking union bounds on the second equation in (\ref{Eq: close}). 

The term $(\rho_{S^c}-G\rho_S)=\frac 1n X_{S^c}'(I-X_S(X_S'X_S)^{-1}X_S')w$ is a linear combination of sub-Gaussian random variables. A direct application of (\ref{eq:hoeff}) yields that $Pr((1/\lam)\| \rho_{S^c}-G\rho_S \linf \geq \gamma/4) \leq \delta(\lambda\gamma,\zeta)$ where $\zeta$ is redefined as maximum of the previous $\zeta$ and $\sigs$. 
%We can choose $\zeta$ to be greater than $\sigs$ by redefining $\zeta$ to be $\max(\zeta,\sigs)$. So, $Pr((1/\lam)\| \rho_{S^c}-G\rho_S \linf \geq \gamma/4) \leq (p-s)C\exp \left( -cn\lam^2\gamma^2\zeta^{-1}\right)$

Without loss of generality, we assume that $\eps_0 \leq 1$. Then with probability greater than $1-\dt$, we can write $||\trho_S-\rho_S||_ \infty + ||\rho_S||_ \infty \leq ||\trho_S-\rho_S||_ \infty + ||\frac 1n X'_Sw||_ \infty + ||\frac 1n X'_SX_S \beta^*_S||_ \infty \leq 2+ B\psi$ for $\eps \leq min(1,\eps_0)$. Combining this with Lemma~\ref{Lem: extrabounds}, we have, with probability at least $1-\delta(\eps,\zeta)$
\[ ||H||_ \infty  \left(1 + \frac 1{\lam} ||\trho_S-\rho_S||_ \infty + \frac 1{\lam} ||\rho_S||_ \infty \right) \leq (1+\frac 1\lambda(2+B\psi)) \; \frac {\eps\phi(2-\gamma)}{(1-\phi\eps)}  \leq \frac \gamma 8\]
for $\eps \leq \eps_0^*$ where $\eps_0^*=\min(\eps_0,\gamma\lambda \phi^{-1}\left(8(2-\gamma)(\lambda+2+B\psi)+\gamma\lambda\right)^{-1})$.

Combining all the probabilities and adjusting for the invertibility probability, for $\lam \leq 4\eps_0/\gamma$ and $\eps \leq min(\eps_0^*,C_{\min}/2)$, we have $Pr(||\tu_{S^c}||_ \infty \geq 1-\gamma/8 ) \leq \delta(\lambda\gamma,\zeta) + \dt$.
\end{proof}

\begin{proof}[Proof of Theorem \ref{Th: gen} Parts (b) and (c)]
Using the expression of $\betah_S$ from Equation~(\ref{Eq: dualfeas vector}), we expand
\begin{align*}
\betah_S-\beta^*_S= &\tSig_{S,S}^{-1}(\trho_S-\rho_S+\frac 1n X_S'X_S \beta_S^*+\frac 1n X_S'w -\lam \tu_S)-\beta_S^*\\
=& F_{S,S}(\trho_S-\rho_S+\frac 1n X_S'X_S \beta_S^* + \frac 1n X_S'w)\\
& + \Sigma_{S,S}^{-1}(\trho_S-\rho_S) + \frac 1n\Sigma_{S,S}^{-1}  X_S'w -\lam \tSig_{S,S}^{-1} \tu_S 
\end{align*}

We analyze each of the terms above separately. From the definition of sub-Gaussian vectors in (\ref{def:sgvec}) we observe that $\frac 1n \Sigma_{S,S}^{-1} X_S'w$ is sub-Gaussian with parameter at most $\sigs C_{\min} / n$. This implies that $||\frac 1n \Sigma_{S,S}^{-1}  X_S'w\linf$ is less than $\lam  / \sqrt{C_{\min}}$ with probability at least $1-\delta(\lambda,\zeta)$. Moreover, as $\tSig=\Sigma+F$, from Lemma~\ref{Lem: extrabounds} we have with probability at least $1-\dt$, for $\eps \leq  min(\eps_0,(2\phi)^{-1})$:
\[ \|\tSig_{S,S} \linf \leq \phi+ ||F\linf \leq  \phi + \phi^2 \eps(1-\phi\eps)^{-1}  \leq 2\phi \] 
The closeness condition for $\trho$ in Equation~(\ref{Eq: close}) implies that $||\trho_S-\rho_S\linf \leq \lam $ with probability at least $1-\delta(\lam,\zeta)$ for $\lam \leq \eps_0$. Following the proof of part (a), we can also conclude that for $\eps \leq \eps_0$, we have $||\trho_S-\rho_S\linf+ || \frac 1n X_S'X_S \beta_S^* \linf  + || \frac 1n X_S'w\linf \leq (2+B\psi)$ with probability at least $1-\delta(\eps,\zeta)$. Therefore, 
\[ || F_{S,S}(\trho_S-\rho_S+\frac 1n X_S'X_S \beta_S^* + \frac 1n X_S'w) \linf < (2+B\psi) \frac {\phi^2 \eps}{1-\phi\eps} \leq \lam \phi \] 
with probability $1-\dt$ for $\eps \leq \lam \phi^{-1} (\lam + 2 + B\psi)^{-1}$. Combining all the probabilities, we have
\begin{align*}
 || \betah_S-\beta^*_S \linf \leq &  || F_{S,S}(\trho_S-\rho_S+\frac 1n X_S'X_S \beta_S^* + \frac 1n X_S'w) \linf \\
& + \phi ||\trho_S-\rho_S\linf + ||\frac 1n \Sigma_{S,S}^{-1}  X_S'w\linf  + 2\lam \phi\\
& \leq \lam \left( 4\phi + \frac 1{\sqrt{C_{\min}}} \right)
\end{align*}
with probability $1-\delta(\lam,\zeta)-\dt$ for $\eps \leq(\eps_0,C_{\min}/2, (2\phi)^{-1}, \lam \phi^{-1} (\lam + 2 + B\psi)^{-1})$ and $\lam \leq \eps_0$.

This proves part (b). If $|\beta^*_{\min}| > \lam ( 4\phi + \frac 1{\sqrt{C_{\min}}})$, then the Lasso estimate is sign consistent proving Part (c).
\end{proof}

\subsection{Proofs of Lemmas~\ref{Lem: addin} and \ref{Lem: multin}}
We assume sub-Gaussian additive or multiplicative measurement errors in Section~\ref{sec:part}. The proofs of Lemmas \ref{Lem: addin} and \ref{Lem: multin} mainly rely on the properties of sub-Gaussian random variables and vectors which can be found in Appendix~\ref{app:subg}.

\begin{proof}[Proof of Lemma \ref{Lem: addin}] 
Let $\Sigma_A=(( \sigma_{a,ij} ))$ and $b_j$ denotes the $j^{th}$ column of any matrix $B$. Then $\hSig_{add,jk}-\Sigma_{jk} = \frac 1n a_j'x_k + \frac 1n a'_kx_j + (\frac 1n a'_ja_k - \sigma_{a,jk})$. Since $\frac 1n ||x_j||_2^2 = 1 $ and the entries of $a_j$ are independent and sub-Gaussian with parameter at most $\taus$ for all $j$, property (\ref{eq:hoeff}) implies that $|(1/n)a_j'x_k|$ and $|(1/n)a_j'x_k|$ are each greater than $\eps/3$ with probability less than $C\exp\left(-cn\eps^2/\taus\right)$. Let $z_i=(a_{ij} ,a_{ik} )'$. Then $z_i$'s are independent sub-Gaussian vectors with parameter at most $\taus$. The tail probability for $\frac 1n a'_ja_k - \sigma_{a,jk}$ can now be made small using Lemma \ref{Lem: sg}. Hence $\hSig_{add}$ satisfies (\ref{Eq: close}) with $\zeta=max(\tau^4,\taus)$ and $\eps_0=c\tau^2$.

We observe that $\trho_{add,j}-\rho_j = \frac 1n a'_jX_S\beta^*_S + \frac 1n a'_jw$. Consequently $|\trho_{add,j}-\rho_j| \leq B \sum_{i=1}^s |\frac 1n a'_jx_i| > \eps/2$ with probability at most $C\exp(-n\eps^2s^{-2}\tau^{-2}B^{-2})$. Letting $z_i=(a_{ij},w_i)$, Lemma~\ref{Lem: sg} can be applied to obtain the tail bound for $\frac 1n a'_jw$. Hence, $\trho_{add}$ satisfies (\ref{Eq: close}) with $\zeta=max(\sigma^4,\tau^4,\taus B^2)$ and $\eps_0 = c \max(\sigs,\taus)$.
\end{proof}

\begin{proof}[Proof of Lemma \ref{Lem: multin}]
The proof once again relies on Lemma \ref{Lem: sg}. Let $\Sigma_M=((\sigma_{m,jk}))$, then 
\begin{align*}
\hSig_{mult,jk}-\Sigma_{jk}  = & \frac 1n \sum_{i=1}^n \frac{ x_{ij} x_{ik}}{\mu_j \mu_k + \sigma_{m,jk}}  (m_{ij} m_{ik} - \mu_j \mu_k - \sigma_{m,jk})\\
= & \frac 1n \sum_{i=1}^n \frac{ x_{ij} x_{ik}}{\mu_j \mu_k + \sigma_{m,jk}} ((m_{ij} - \mu_j)(m_{ik} - \mu_k) - \sigma_{m,jk})\\
& +\frac 1n \sum_{i=1}^n \frac{ x_{ij} x_{ik}}{\mu_j \mu_k + \sigma_{m,jk}} (\mu_j(m_{ik} - \mu_k)+\mu_k(m_{ij} - \mu_j))
\end{align*}
Using the regularity conditions in Equation \ref{Eq: multcond}, we have, 
\begin{align}\label{eq:split}
  |\hSig_{mult,jk}- \Sigma_{jk}| \leq & \frac 1{M_{\min}} |(1/n) \sum_{i=1}^n x_{ij} x_{ik} ((m_{ij}- \mu_j) ( m_{ik} -\mu_k) - \sigma_{m,jk})|\\
 & + \frac {\mu_{\max}}{M_{\min}} | (1/n)  \sum_{i=1}^n x_{ij}x_{ik} (m_{ik} - \mu_k)| \nonumber \\
 & + \frac {\mu_{\max}}{M_{\min}} |(1/n) \sum_{i=1}^n x_{ij}x_{ik}(m_{ij} - \mu_j)| \nonumber
\end{align}
We denote the three terms on the right hand side of (\ref{eq:split}) by $T_1$, $T_2$ and $T_3$ respectively. Note that, if $v=(v_1,v_2,\ldots,v_n)$ where $v_i=x_{ij}x_{jk}$, then $||v||_\infty \leq X_{\max}^2$ As, the errors are once again sub-Gaussian, using Lemma \ref{Lem: sg}, we see that for $ \zeta=\max(\tau^4X_{\max}^4/M_{\min}^2, \tau^2X^2_{\max}\mu_{\max}^2/M_{\min}^2)$ and $\eps \leq c\taus X^2_{\max} /M_{\min}$ we have:
\begin{align*}
Pr(T_1 \geq \eps) & \leq C\exp\left(-cn\eps^2\zeta^{-1}\right) \mbox{ for }
\end{align*}
The terms $T_2$ and $T_3$ can be similarly bounded using property (\ref{eq:hoeff}).
This proves that $\hSig_{mult}$ satisfies (\ref{Eq: close}). We now show that $\trho_{mult}$ also satisfies (\ref{Eq: close}). Recall that

$\trho_{mult,j}-\rho_j = (1/n) (z_j-\mu_jx_j)'y/\mu_j$. As $y=X_S\beta^*_S+w$, we have
\begin{align*}
 |\trho_{mult,j}-\rho_j| \leq & \frac 1{\mu_{\min}} \sum_{k=1}^s |\frac 1n (z_j-\mu_jx_j)' x_k \beta^*_k|+ \frac 1{\mu_{\min}} |\frac 1n (z_j-\mu_jx_j)' w| \\
  \leq & \frac B{\mu_{\min}}  \sum_{k=1}^s |(1/n) \sum_{i=1}^n x_{ij}x_{ik} (m_{ij}-\mu_j)| \\
  & + \frac 1{\mu_{\min}} |(1/n)\sum_{i=1}^n x_{ij}w_j(m_{ij}-\mu_j)| 
\end{align*}
Using Lemma \ref{Lem: sg}, we have for $\zeta= X_{\max}^2 \max(\tau^2B^2/, \tau^4, \sigma^4 ) /\mu_{\min}^2  $ and $\eps \leq c X_{max} \max (\taus ,\sigs) /\mu_{\min} $:
\begin{align*}
& Pr(\frac 1{\mu_{\min}} |(1/n)\sum_{i=1}^n x_{ij}w_j(m_{ij}-\mu_j)| \geq \eps/2)  \leq C\exp\left(-cn\eps^2\zeta^{-1}\right)\\
& Pr(\frac B{\mu_{\min}} |(1/n) \sum_{i=1}^n x_{ij}x_{ik} (m_{ij}-\mu_j)| \geq \eps/{2s})  \leq C\exp\left(-cn\eps^2s^{-2}\zeta^{-1}\right)  
\end{align*}
where the last inequality follows from property (\ref{eq:hoeff}).
\end{proof} 

\section*{Acknowledgment} The authors thank Dr. Po-Ling Loh for sharing  R and Matlab codes for computing the NCL estimator. Zou is partially supported by a NSF grant.

\appendix

\section{Algorithm for finding the Nearest positive semi-definite matrix}\label{app:admm}
We use an alternating direction method of multipliers to solve for
\begin{equation}\label{eq:npsd}
\hat A=\underset{A \geq \eps I}{\arg \min } ||A-\hSig||_ {\max}
\end{equation}
for any $\eps > 0$. We introduce an additional variable $B$ and an equality constraint $B=A-\hSig$ to rewrite the optimization problem in (\ref{eq:npsd}) as
\begin{equation}\label{eq:reparam}
(\hat A, \hat B) = \underset{A \geq \eps I, \; B=A-\hSig}{\arg \min } ||B||_ {\max}
\end{equation}

To solve (\ref{eq:reparam}) we will minimize the augmented Lagrangian function:
\begin{equation}\label{eq:lagrange}
f(A,B,\Lambda)=\frac 12 ||B||_ {\max}-\langle \Lambda, A-B-\hSig \rangle + \frac 1{2\mu} ||A-B-\hSig||_F^2
\end{equation}
where $\mu$ is some penalty parameter, $\Lambda$ is the Lagrangian matrix and $\langle \cdot , \cdot \rangle$ denotes the matrix inner product which induces the Frobenius norm $||\cdot||_F$. We solve for the minimizer of $f(A,B,\Lambda)$ iteratively using the following three steps at the $i^{th}$ iteration:
\begin{align}\label{eq:steps}
A \mbox{ step: } & A_{i+1}= \underset{A \geq \eps I}{\arg \min} f(A,B_i,\Lambda_i)  \nonumber\\
B \mbox{ step: } & B_{i+1}= \underset{B}{\arg \min} f(A_{i+1},B,\Lambda_i) \\
\Lambda \mbox{ step: } & \Lambda_{i+1} = \Lambda_i - \frac {A_{i+1}-B_{i+1}-\hSig}{\mu} \nonumber
\end{align}
We now provide the closed-form solutions for the first two steps in Equation~(\ref{eq:steps}). The A step can be simplified as:
\begin{align*}
	 \underset{A \geq \eps I}{\arg \min} f(A,B_i,\Lambda_i) & = 
	\underset{A \geq \eps I}{\arg \min} \; 1/(2\mu) ||A-B_i-\hSig||_F^2-\langle \Lambda_i, A \rangle \\
 	& = \underset{A \geq \eps I}{\arg \min} ||A-B_i-\hSig-\mu \Lambda_i||_F^2
\end{align*}
The unconstrained solution for the A-step is $B_i+\hSig+\mu \Lambda_i$. Let for any symmetric matrix $Z$, $Z_\eps$ denote the projection of $Z$ into the space of matrices with eigen values greater than $\eps$. If $Z=\sum_j \lambda_j p_j p_j' $ denote the spectral decomposition of $Z$, then we have $Z_\eps = \sum_j max(\lambda_j,\eps) p_j p_j' $. Hence, the solution for the A-step is given by,
\begin{equation}\label{eq:astep}
A_{i+1}=(B_i+\hSig+\mu \Lambda_i)_\eps
\end{equation}
The B-step is equivalent to:
\begin{align}
	 & \underset{B}{\arg \min} \; \frac 12 ||B||_{\max} + \frac 1{2\mu} ||B-(A_{i+1}-\hSig)||_F^2 - \langle -\Lambda_i, B\rangle \\
	 & = \underset{B}{\arg \min} \; ||B-(A_{i+1}-\hSig-\mu\Lambda_i)||_F^2 +  \mu ||B||_{\max}  \nonumber
\end{align}
Let for any symmetric matrix $M$, $vec_L(M)$ denote the vector containing the lower half elements (including the diagonal) of $M$. Since, $vec_l$ is an injective mapping, we can define an inverse mapping $mat_l(x)$ such that $mat_l(vec_l(M))=M$ for any symmetric matrix $M$. The solution to the B-step is given by 
\[ B_{i+1}= mat_l(vec_l(A_{i+1}-\hSig-\mu \Lambda_i) - \ell_1(vec_l(A_{i+1}-\hSig-\mu \Lambda_i, \mu)))\]
where for any vector $x$ and $\mu > 0$, $\ell_1(x,\mu)$ is the projection of $x$ into the $\ell_1$ ball of radius $\mu$. The algorithm to calculate $\ell_1(x,\mu)$ is provided in \cite{duchi08}.

\begin{algorithm}
\caption{ADMM algorithm for finding the nearest positive semi-definite matrix}
\label{algo:admm}
\begin{algorithmic}[1]
\\Input $\mu$ and the initial values $B_0$ and $\Lambda_0$
\\At the $i^{th}$ step update:
\begin{enumerate}[label=2.{\arabic*}:]
\item (Step $A$) $A_{i+1}=(B_i+\hSig+\mu \Lambda_i)_\eps$
\item (Step $B$) $B_{i+1}=mat_l(vec_l(A_{i+1}-\hSig-\mu \Lambda_i) - \ell_1(vec_l(A_{i+1}-\hSig-\mu \Lambda_i, \mu)))$
\item (Step $\Lambda$) $\Lambda_{i+1} = \Lambda_i - \frac {A_{i+1}-B_{i+1}-\hSig}{\mu}$
\end{enumerate}
\\ Repeat Step 2 till convergence
\end{algorithmic}
\end{algorithm}

\section{Sub-Gaussian Random Variables}\label{app:subg}
In our analysis of the CoCoLasso estimate, we have assumed that the errors $w$ are independent and identically distributed sub-Gaussian random variables with parameter $\tau^2$. In this Section, we summarize some useful definitions and properties of sub-Gaussian random variables. 
\begin{adefn} (Sub-Gaussian random variables, \citeauthor {roman}  \citeyear{roman} \cite{roman}) A random variable $Z$ is said to be sub-Gaussian if there exists a finite $\kappa > 0$ such that $\kappa= \underset{p \geq 1}{\sup} \; p^{-1/2} (E|X|^p)^ \frac 1p$. $\kappa$ is referred to as the sub-Gaussian norm of $Z$ denoted by $||Z||_\phi$
\end{adefn}
Equivalently, a sub-Gaussian random variable $Z$ satisfies the following tail probability bounds:
\begin{equation}\label{eq:subgtail}
P(|Z| > t) \leq 2\exp(- t^2 /2\taus) \mbox{ for all } t>0
\end{equation}
To avoid ambiguity, we refer to the sub-Gaussian parameter of $Z$ as the smallest $\taus$ satisfying (\ref{eq:subgtail}). Following \cite[Lemma 5.5]{roman} we observe that there exists universal constants $m$ and $M$ such that $m||Z||^2_\phi \leq \taus \leq M||Z||^2_\phi$. We note that if $w=(w_1,w_2,\ldots,w_n)'$ that $w_i$'s are independent zero-centered sub-Gaussian random variables, then weighted sums of $w_i$ are also sub-Gaussian and satisfy an useful property \cite[Lemma 5.9]{roman}:
\begin{equation}\label{eq:hoeff}
||v'w||^2_\phi \leq K ||v||_2^2 \; \underset{i}{\max} (|| w_i||_\phi^2)
\end{equation}
where $K$ is an absolute constant. The tail-probability characterization in (\ref{eq:subgtail}) enables defining sub-Gaussian random vectors in the following sense:
\begin{adefn}({Sub-Gaussian random vectors, \citeauthor{covmat} \citeyear{covmat} \cite{covmat}})\label{def:sgvec}
A random vector $w$ is said to be sub-Gaussian if there exists $\tau > 0$ such that $Pr(|v'(w-E(w))| > t) \leq 2\exp(-\frac {t^2}{2\taus})$ for all $t>0$ and $||v||_2=1$.
\end{adefn}

From property (\ref{def:sgvec}) it is clear that if $w=(w_1,w_2,\ldots,w_n)'$ is a sub-Gaussian vector with parameter $\taus$, then each $w_i$ is also sub-Gaussian with parameter at most $\taus$. Conversely if $w_i$'s are independent and sub-Gaussian random variables with parameter $\taus_i$, then $w=(w_1,w_2,\ldots,w_n)$ is a sub-Gaussian vector with parameter at most $\taus \leq (KM/m) (\max \taus_i)$. We now state and prove another useful result for correlated sub-Gaussian sequences:
\begin{alemma}\label{Lem: sg}
Let $z_i =(x_i,y_i)'$ denote independent and identically distributed vectors with zero mean, covariance $\Sigma=((\sigma_{ij}))$ and  sub-Gaussian parameter $\taus$. Then there exists absolute constants $C$ and $c$ such that, for every $\eps \leq c\taus\|a\linf$, we have:
\begin{equation}\label{Eq: sg2}
Pr(\frac 1n |\sum_{i=1}^n a_i(x_iy_i - \sigma_{12})| \geq \eps) \leq C \exp\left(-\frac {nc\eps^2}{\tau^4||a\linf^2} \right)
\end{equation}
\end{alemma}

\begin{proof}
\begin{align*} (1/n)\sum_{i=1}^n a_i(x_iy_i - \sigma_{12})  = \frac 1{4n}\sum_{i=1}^n & a_i\left((x_i+y_i)^2 - (\sigma_{11} + \sigma_{22}+2\sigma_{12})\right) \\
- \frac 1{4n}\sum_{i=1}^n & a_i \left( (x_i-y_i)^2 - (\sigma_{11}+\sigma_{22}-2\sigma_{12}) \right)\\
= \frac1{2n} \sum_{i=1}^n a_i\left( (v_1'z_i)^2 - E((v_1'z_1)^2) \right) & - \frac1{2n} \sum_{i=1}^n a_i\left( (v_2'z_i)^2 - E((v_2'z_1)^2) \right)
\end{align*}
where $v_1=  (1/\sqrt 2, 1/\sqrt 2)'$ and $v_1=  (1/\sqrt 2, -1/\sqrt 2)'$. 
As $||v_k||=1$, $v_k'z_1$ is sub-Gaussian with parameter at most $\taus$ for $k=1,2$. Using the relationship between sub-Gaussian and sub-exponential random variables in \cite[Lemma 5.14 and Remark 5.18]{roman}, we see that, for $k=1,2$, $(v_k'z_i)^2 - E((v_k'z_1)^2)$ is sub-exponential with parameter at most $c\taus$ where $c$ is an absolute constant. As a result $t_i=a_i((v_1'z_i)^2 - E((v_1'z_1)^2)$ is sub-exponential with parameter at most  $c\taus||a\linf$. A direct application of \cite[Corollary 5.17]{roman} now yields for $\eps \leq c \taus ||a\linf$, \[ Pr\left(\frac 1{2n}|\sum _{i=1}^n t_i | \geq \eps \right) \leq C\exp\left(-\frac {nc\eps^2}{\tau^4 ||a\linf^2} \right)\]

%We, know that $Pr( |v^T((1/n) \sum_{i=1}^n z_iz_i^T - \Sigma ) v | \geq \eps) \leq C\exp\left(-nc\eps^2/\tau_1 \right)$ for all $v$ such that $\|v\|_2=1$ and $\eps \leq 1/\tau_1$ (See, for example, proof of lemma 3 in \cite{covmat}). As, $\|v_1 \|_2=\|v_2\|_2=1$ and $|(1/n)\sum_{i=1}^n x_iy_i - \sigma_{12}| \leq |\frac1{2n} v_1^T(\sum_{i=1}^n z_iz_i^T -\Sigma)v_1|+|\frac1{2n} v_2^T(\sum_{i=1}^n z_iz_i^T -\Sigma)v_2|$, the result follows.
\end{proof}
\bibliographystyle{imsart-nameyear}
\bibliography{COCObib}

\begin{thebibliography}{30}
% BibTex style file: imsart-nameyear.bst, 2013-01-28
% Default style options (sort=1,type=nameyear).
% Used options (sort=1,type=nameyear).

\bibitem[\protect\citeauthoryear{Belloni, Rosenbaum and
  Tsybakov}{2014a}]{linandconic}
\begin{barticle}[author]
\bauthor{\bsnm{Belloni},~\bfnm{A.}\binits{A.}},
  \bauthor{\bsnm{Rosenbaum},~\bfnm{M.}\binits{M.}} \AND
  \bauthor{\bsnm{Tsybakov},~\bfnm{A.~B.}\binits{A.~B.}}
(\byear{2014}a).
\btitle{Linear and Conic Programming Estimators in High-Dimensional
  Errors-in-variables Models}.
\bjournal{http://arxiv.org/abs/1408.0241}.
\end{barticle}
\endbibitem

\bibitem[\protect\citeauthoryear{Belloni, Rosenbaum and
  Tsybakov}{2014b}]{l1l2linf}
\begin{barticle}[author]
\bauthor{\bsnm{Belloni},~\bfnm{A.}\binits{A.}},
  \bauthor{\bsnm{Rosenbaum},~\bfnm{M.}\binits{M.}} \AND
  \bauthor{\bsnm{Tsybakov},~\bfnm{A.~B.}\binits{A.~B.}}
(\byear{2014}b).
\btitle{An {$\ell_1,\ell_2,\ell_{\infty}$}-Regularization Approach to
  High-Dimensional Errors-in-variables Models}.
\bjournal{http://arxiv.org/abs/1412.7216}.
\end{barticle}
\endbibitem

\bibitem[\protect\citeauthoryear{Benjamini and Speed}{2012}]{thru}
\begin{barticle}[author]
\bauthor{\bsnm{Benjamini},~\bfnm{Y.}\binits{Y.}} \AND
  \bauthor{\bsnm{Speed},~\bfnm{T.~P.}\binits{T.~P.}}
(\byear{2012}).
\btitle{Estimation and Correction for GC-content Bias in High Throughput
  Sequencing}.
\bjournal{Nucleic Acids Research}
\bvolume{40}
\bpages{e72}.
\end{barticle}
\endbibitem

\bibitem[\protect\citeauthoryear{Boyd et~al.}{2011}]{boyd_admm}
\begin{barticle}[author]
\bauthor{\bsnm{Boyd},~\bfnm{Stephen}\binits{S.}},
  \bauthor{\bsnm{Parikh},~\bfnm{Neal}\binits{N.}},
  \bauthor{\bsnm{Chu},~\bfnm{Eric}\binits{E.}},
  \bauthor{\bsnm{Peleato},~\bfnm{Borja}\binits{B.}} \AND
  \bauthor{\bsnm{Eckstein},~\bfnm{Jonathan}\binits{J.}}
(\byear{2011}).
\btitle{Distributed Optimization and Statistical Learning via the Alternating
  Direction Method of Multipliers}.
\bjournal{Foundations and Trends in Machine Learning}
\bvolume{3}
\bpages{1--122}.
\bdoi{10.1561/2200000016}
\end{barticle}
\endbibitem

\bibitem[\protect\citeauthoryear{Buhlmann and van~de Geer}{2011}]{buhl}
\begin{bbook}[author]
\bauthor{\bsnm{Buhlmann},~\bfnm{Peter}\binits{P.}} \AND
  \bauthor{\bparticle{van~de} \bsnm{Geer},~\bfnm{Sara~A.}\binits{S.~A.}}
(\byear{2011}).
\btitle{Statistics for High-Dimensional Data: Methods, Theory and
  Applications},
\bedition{1st} ed.
\bpublisher{Springer Publishing Company, Incorporated}.
\end{bbook}
\endbibitem

\bibitem[\protect\citeauthoryear{Cai, Zhang and Zhou}{2010}]{covmat}
\begin{barticle}[author]
\bauthor{\bsnm{Cai},~\bfnm{T.~T.}\binits{T.~T.}},
  \bauthor{\bsnm{Zhang},~\bfnm{C.}\binits{C.}} \AND
  \bauthor{\bsnm{Zhou},~\bfnm{H.~H.}\binits{H.~H.}}
(\byear{2010}).
\btitle{Optimal Rates of Convergence for Covariance Matrix Estimation}.
\bjournal{Annals of Statistics}
\bvolume{38}
\bpages{2118–-2144}.
\end{barticle}
\endbibitem

\bibitem[\protect\citeauthoryear{Cand\`{e}s and Tao}{2007}]{dantzig}
\begin{barticle}[author]
\bauthor{\bsnm{Cand\`{e}s},~\bfnm{E.}\binits{E.}} \AND
  \bauthor{\bsnm{Tao},~\bfnm{T.}\binits{T.}}
(\byear{2007}).
\btitle{The Dantzig Selector: Statistical Estimation When p Is Much Larger Than
  n}.
\bjournal{Annals of Statistics}
\bvolume{35}
\bpages{2313--2351}.
\end{barticle}
\endbibitem

\bibitem[\protect\citeauthoryear{Duchi et~al.}{2008}]{duchi08}
\begin{binproceedings}[author]
\bauthor{\bsnm{Duchi},~\bfnm{John}\binits{J.}},
  \bauthor{\bsnm{Shalev-Shwartz},~\bfnm{Shai}\binits{S.}},
  \bauthor{\bsnm{Singer},~\bfnm{Yoram}\binits{Y.}} \AND
  \bauthor{\bsnm{Chandra},~\bfnm{Tushar}\binits{T.}}
(\byear{2008}).
\btitle{Efficient Projections Onto the L1-ball for Learning in High
  Dimensions}.
In \bbooktitle{Proceedings of the 25th International Conference on Machine
  Learning}.
\bseries{ICML '08}
\bpages{272--279}.
\bpublisher{ACM}, \baddress{New York, NY, USA}.
\bdoi{10.1145/1390156.1390191}
\end{binproceedings}
\endbibitem

\bibitem[\protect\citeauthoryear{Efron, Hastie and
  Tibshirani}{2007}]{disc_dantzig_efron}
\begin{barticle}[author]
\bauthor{\bsnm{Efron},~\bfnm{B.}\binits{B.}},
  \bauthor{\bsnm{Hastie},~\bfnm{T.}\binits{T.}} \AND
  \bauthor{\bsnm{Tibshirani},~\bfnm{R}\binits{R.}}
(\byear{2007}).
\btitle{Discussion: The Dantzig Selector: Statistical Estimation When p Is Much
  Larger Than n}.
\bjournal{The Annals of Statistics}
\bvolume{35}
\bpages{2358-–2364}.
\end{barticle}
\endbibitem

\bibitem[\protect\citeauthoryear{Efron et~al.}{2004}]{lars}
\begin{barticle}[author]
\bauthor{\bsnm{Efron},~\bfnm{Bradley}\binits{B.}},
  \bauthor{\bsnm{Hastie},~\bfnm{Trevor}\binits{T.}},
  \bauthor{\bsnm{Johnstone},~\bfnm{Iain}\binits{I.}} \AND
  \bauthor{\bsnm{Tibshirani},~\bfnm{Robert}\binits{R.}}
(\byear{2004}).
\btitle{Least Angle Regression}.
\bjournal{Annals of Statistics}
\bvolume{32}
\bpages{407--499}.
\end{barticle}
\endbibitem

\bibitem[\protect\citeauthoryear{Fan and Li}{2001}]{scad}
\begin{barticle}[author]
\bauthor{\bsnm{Fan},~\bfnm{J.}\binits{J.}} \AND
  \bauthor{\bsnm{Li},~\bfnm{R.}\binits{R.}}
(\byear{2001}).
\btitle{Variable Selection via Nonconcave Penalized Likelihood and Its Oracle
  Properties}.
\bjournal{Journal of the American Statistical Association}
\bvolume{96}
\bpages{1348–-1360}.
\end{barticle}
\endbibitem

\bibitem[\protect\citeauthoryear{Fan and Li}{2006}]{fanli06}
\begin{binbook}[author]
\bauthor{\bsnm{Fan},~\bfnm{Jianqing}\binits{J.}} \AND
  \bauthor{\bsnm{Li},~\bfnm{Runze}\binits{R.}}
(\byear{2006}).
\btitle{Statistical challenges with high dimensionality: Feature selection in
  knowledge discovery}
In \bbooktitle{International Congress of Mathematicians. Vol. III}
\bpages{595{\textendash}622}.
\bpublisher{Eur. Math. Soc., Z{\"u}rich}.
\end{binbook}
\endbibitem

\bibitem[\protect\citeauthoryear{Fan and Lv}{2010}]{fanlv10}
\begin{barticle}[author]
\bauthor{\bsnm{Fan},~\bfnm{J.}\binits{J.}} \AND
  \bauthor{\bsnm{Lv},~\bfnm{J.}\binits{J.}}
(\byear{2010}).
\btitle{A Selective Overview of Variable Selection in High Dimensional Feature
  Space}.
\bjournal{Statistica Sinica}
\bvolume{20}
\bpages{101--148}.
\end{barticle}
\endbibitem

\bibitem[\protect\citeauthoryear{Friedman, Hastie and
  Tibshirani}{2010}]{glmnet}
\begin{barticle}[author]
\bauthor{\bsnm{Friedman},~\bfnm{Jerome~H.}\binits{J.~H.}},
  \bauthor{\bsnm{Hastie},~\bfnm{Trevor}\binits{T.}} \AND
  \bauthor{\bsnm{Tibshirani},~\bfnm{Rob}\binits{R.}}
(\byear{2010}).
\btitle{Regularization Paths for Generalized Linear Models via Coordinate
  Descent}.
\bjournal{Journal of Statistical Software}
\bvolume{33}
\bpages{1--22}.
\end{barticle}
\endbibitem

\bibitem[\protect\citeauthoryear{Hastie, Tibshirani and
  Friedman}{2011}]{hastibs}
\begin{bbook}[author]
\bauthor{\bsnm{Hastie},~\bfnm{Trevor}\binits{T.}},
  \bauthor{\bsnm{Tibshirani},~\bfnm{Robert}\binits{R.}} \AND
  \bauthor{\bsnm{Friedman},~\bfnm{Jerome}\binits{J.}}
(\byear{2011}).
\btitle{The Elements of Statistical Learning},
\bedition{second} ed.
\bseries{Springer Series in Statistics}.
\bpublisher{Springer New York Inc.}, \baddress{New York, NY, USA}.
\end{bbook}
\endbibitem

\bibitem[\protect\citeauthoryear{Loh and Wainwright}{2012}]{loh12}
\begin{barticle}[author]
\bauthor{\bsnm{Loh},~\bfnm{P.~L.}\binits{P.~L.}} \AND
  \bauthor{\bsnm{Wainwright},~\bfnm{M.~J.}\binits{M.~J.}}
(\byear{2012}).
\btitle{High-dimensional Regression with Noisy and Missing Data: Provable
  Guarantees with Non-convexity}.
\bjournal{Annals of Statistics}
\bvolume{40}
\bpages{1637--1664}.
\end{barticle}
\endbibitem

\bibitem[\protect\citeauthoryear{Mai, Zou and Yuan}{2012}]{huisda}
\begin{barticle}[author]
\bauthor{\bsnm{Mai},~\bfnm{Q.}\binits{Q.}},
  \bauthor{\bsnm{Zou},~\bfnm{H.}\binits{H.}} \AND
  \bauthor{\bsnm{Yuan},~\bfnm{M.}\binits{M.}}
(\byear{2012}).
\btitle{{A Direct Approach to Sparse Discriminant Analysis in Ultra-high
  Dimensions}}.
\bjournal{Biometrika}
\bvolume{99}
\bpages{29--42}.
\end{barticle}
\endbibitem

\bibitem[\protect\citeauthoryear{Purdom and Holmes}{2005}]{generror}
\begin{barticle}[author]
\bauthor{\bsnm{Purdom},~\bfnm{E.}\binits{E.}} \AND
  \bauthor{\bsnm{Holmes},~\bfnm{S.~P.}\binits{S.~P.}}
(\byear{2005}).
\btitle{Error Distribution for Gene Expression Data}.
\bjournal{Statistical Applications in Genetics and Molecular Biology}
\bvolume{4}.
\bdoi{10.2202/1544-6115.1070}
\end{barticle}
\endbibitem

\bibitem[\protect\citeauthoryear{Rosenbaum and Tsybakov}{2010}]{rose10}
\begin{barticle}[author]
\bauthor{\bsnm{Rosenbaum},~\bfnm{M.}\binits{M.}} \AND
  \bauthor{\bsnm{Tsybakov},~\bfnm{A.~B.}\binits{A.~B.}}
(\byear{2010}).
\btitle{Sparse Recovery under Matrix Uncertainty}.
\bjournal{Annals of Statistics}
\bvolume{38}
\bpages{2620--2651}.
\end{barticle}
\endbibitem

\bibitem[\protect\citeauthoryear{Rosenbaum and Tsybakov}{2013}]{rose13}
\begin{barticle}[author]
\bauthor{\bsnm{Rosenbaum},~\bfnm{M.}\binits{M.}} \AND
  \bauthor{\bsnm{Tsybakov},~\bfnm{A.~B.}\binits{A.~B.}}
(\byear{2013}).
\btitle{Improved Matrix Uncertainty Selector}.
\bjournal{IMS Collections. From probability to statistics and back: High
  dimensional models and processes}
\bvolume{9}
\bpages{276--290}.
\end{barticle}
\endbibitem

\bibitem[\protect\citeauthoryear{Slijepcevic, Megerian and
  Potkonjak}{2002}]{senserror}
\begin{barticle}[author]
\bauthor{\bsnm{Slijepcevic},~\bfnm{S.}\binits{S.}},
  \bauthor{\bsnm{Megerian},~\bfnm{S.}\binits{S.}} \AND
  \bauthor{\bsnm{Potkonjak},~\bfnm{M.}\binits{M.}}
(\byear{2002}).
\btitle{Location Errors in Wireless Embedded Sensor Networks: Sources, Models,
  and Effects on Applications}.
\bjournal{Mobile Computing and Communications Review}
\bvolume{6}
\bpages{67–-78}.
\end{barticle}
\endbibitem

\bibitem[\protect\citeauthoryear{S{\o}rensen, Frigessi and
  Thoresen}{2013}]{oys13}
\begin{barticle}[author]
\bauthor{\bsnm{S{\o}rensen},~\bfnm{{\O}.}\binits{{\O}.}},
  \bauthor{\bsnm{Frigessi},~\bfnm{A.}\binits{A.}} \AND
  \bauthor{\bsnm{Thoresen},~\bfnm{M}\binits{M.}}
(\byear{2013}).
\btitle{{Measurement Error in LASSO: Impact and Likelihood Bias Correction}}.
\bjournal{Statistica sinica}
\bvolume{23}
\bpages{Preprint}.
\end{barticle}
\endbibitem

\bibitem[\protect\citeauthoryear{Tibshirani}{1994}]{lasso}
\begin{barticle}[author]
\bauthor{\bsnm{Tibshirani},~\bfnm{R.}\binits{R.}}
(\byear{1994}).
\btitle{Regression Shrinkage and Selection Via the Lasso}.
\bjournal{Journal of the Royal Statistical Society, Series B}
\bvolume{58}
\bpages{267--288}.
\end{barticle}
\endbibitem

\bibitem[\protect\citeauthoryear{Tibshirani et~al.}{2005}]{fusedlasso}
\begin{barticle}[author]
\bauthor{\bsnm{Tibshirani},~\bfnm{Robert}\binits{R.}},
  \bauthor{\bsnm{Saunders},~\bfnm{Michael}\binits{M.}},
  \bauthor{\bsnm{Rosset},~\bfnm{Saharon}\binits{S.}},
  \bauthor{\bsnm{Zhu},~\bfnm{Ji}\binits{J.}} \AND
  \bauthor{\bsnm{Knight},~\bfnm{Keith}\binits{K.}}
(\byear{2005}).
\btitle{Sparsity and smoothness via the fused lasso}.
\bjournal{Journal of the Royal Statistical Society Series B}
\bpages{91--108}.
\end{barticle}
\endbibitem

\bibitem[\protect\citeauthoryear{van~de Geer and Buhlmann}{2009}]{vdg09}
\begin{barticle}[author]
\bauthor{\bparticle{van~de} \bsnm{Geer},~\bfnm{S.~A.}\binits{S.~A.}} \AND
  \bauthor{\bsnm{Buhlmann},~\bfnm{P.}\binits{P.}}
(\byear{2009}).
\btitle{{On the Conditions Used to Prove Oracle Results for the Lasso}}.
\bjournal{Electronic Journal of Statistics}
\bvolume{3}
\bpages{1360--1392}.
\end{barticle}
\endbibitem

\bibitem[\protect\citeauthoryear{Vershynin}{2012}]{roman}
\begin{binbook}[author]
\bauthor{\bsnm{Vershynin},~\bfnm{R.}\binits{R.}}
(\byear{2012}).
\btitle{Introduction to the Non-asymptotic Analysis of Random Matrices.
  Compressed Sensing}
In \bbooktitle{Compressed Sensing}
\bpages{210--268}.
\bpublisher{Cambridge Univ. Press, Cambridge, UK}.
\end{binbook}
\endbibitem

\bibitem[\protect\citeauthoryear{Wainwright}{2009}]{wainsgn}
\begin{barticle}[author]
\bauthor{\bsnm{Wainwright},~\bfnm{M.~J.}\binits{M.~J.}}
(\byear{2009}).
\btitle{Sharp Thresholds for High-Dimensional and Noisy Sparsity Recovery Using
  $l_1$-Constrained Quadratic Programming (Lasso)}.
\bjournal{IEEE Transactions on Information Theory}
\bvolume{55}
\bpages{2183--2202}.
\end{barticle}
\endbibitem

\bibitem[\protect\citeauthoryear{Zhao and Yu}{2006}]{zhaoyu}
\begin{barticle}[author]
\bauthor{\bsnm{Zhao},~\bfnm{Peng}\binits{P.}} \AND
  \bauthor{\bsnm{Yu},~\bfnm{Bin}\binits{B.}}
(\byear{2006}).
\btitle{On Model Selection Consistency of Lasso}.
\bjournal{Journal of Machine Learning Research}
\bvolume{7}
\bpages{2541--2563}.
\end{barticle}
\endbibitem

\bibitem[\protect\citeauthoryear{Zou}{2006}]{adalasso}
\begin{barticle}[author]
\bauthor{\bsnm{Zou},~\bfnm{H.}\binits{H.}}
(\byear{2006}).
\btitle{The Adaptive Lasso and Its Oracle Properties}.
\bjournal{Journal of the American Statistical Association}
\bvolume{101}
\bpages{1418--1429}.
\end{barticle}
\endbibitem

\bibitem[\protect\citeauthoryear{Zou and Hastie}{2005}]{enet}
\begin{barticle}[author]
\bauthor{\bsnm{Zou},~\bfnm{H.}\binits{H.}} \AND
  \bauthor{\bsnm{Hastie},~\bfnm{T.}\binits{T.}}
(\byear{2005}).
\btitle{Regularization and Variable Selection Via the Elastic Net}.
\bjournal{Journal of the Royal Statistical Society, Series B}
\bvolume{67}
\bpages{301--320}.
\end{barticle}
\endbibitem

\end{thebibliography}

\end{document}